\documentclass[11pt,a4paper,leqno]{amsart}
\usepackage{amssymb,amsmath,amsthm}

\usepackage[mathscr]{eucal} 
\usepackage{graphicx}
\usepackage[hidelinks]{hyperref}
\usepackage{xcolor}
\hypersetup{
 colorlinks,
 linkcolor={red!50!black},
 citecolor={blue!50!black},
 urlcolor={blue!80!black}
}

\usepackage{pdfsync}
\usepackage{enumerate} 
\usepackage{dsfont}
\usepackage{esint}

\usepackage{relsize} 
\usepackage{bigints}

\allowdisplaybreaks

\theoremstyle{plain}
\newtheorem{theorem}{Theorem}[section]

\newtheorem{assumption}[theorem]{Assumption}
\newtheorem{lemma}[theorem]{Lemma}
\newtheorem{corollary}[theorem]{Corollary}
\newtheorem{proposition}[theorem]{Proposition}
\theoremstyle{remark}
\newtheorem{remark}[theorem]{Remark}

\numberwithin{equation}{section}

\def\R{{\mathbb R}}
\def\N{{\mathbb N}}

 \def\dd{\mathrm d}
\newcommand{\loc}{\rm loc}
\DeclareMathOperator{\diver}{div}
\DeclareMathOperator{\curl}{curl}
\newcommand{\dist}{\operatorname{dist}}
\newcommand{\supp}{\operatorname{supp}}

\renewcommand{\leq}{\leqslant}
\renewcommand{\geq}{\geqslant}

\def\Tend#1#2{\mathop{\longrightarrow}\limits_{#1\rightarrow#2}}

\def\cal#1{\mathcal{#1}}
\def\mb#1{\boldsymbol{#1}}

\def\eps{\varepsilon}

\newcommand{\veps}{v_{\eps}}

\newcommand{\ueps}{v_{\eps}}

\newcommand{\weps}{\omega_{\eps}}

\newcommand{\Ie}{\mathcal{I}_{\eps}}

\newcommand{\lne}{|\ln \eps|}

\date\today

\title[Vortex dynamics for the lake equations]{Dynamics of several point vortices for the lake equations}

\author[L.E. Hientzsch]{Lars Eric Hientzsch}

\author[C. Lacave]{Christophe Lacave}

\author[E. Miot]{Evelyne Miot}

\email{\newline lhientzsch@math.uni-bielefeld.de \newline christophe.lacave@univ-grenoble-alpes.fr \newline evelyne.miot@univ-grenoble-alpes.fr}

\address[L.E. Hientzsch]{Fakult\"at f\"ur Mathematik, Universit\"at Bielefeld, Postfach 10 01 31, 33501 Bielefeld, Germany
}
\address[C. Lacave $\&$ E. Miot]{Univ. Grenoble Alpes, CNRS, Institut Fourier, F-38000 Grenoble, France.}

\begin{document}
\begin{abstract} 
The global asymptotic dynamics of point vortices for the lake equations is rigorously derived. Vorticity that is initially sharply concentrated around $N$ distinct vortex centers is proven to remain concentrated for all times. Specifically, we prove weak concentration of the vorticity and in addition strong concentration in the direction of the steepest ascent of the depth function. As a consequence, we obtain the motion law of point vortices following at leading order the level lines of the depth function. The lack of strong localization in the second direction is linked to the vortex filamentation phenomena. The main result allows for any fixed number of vortices and general assumptions on the concentration property of the initial data to be considered. No further properties such as a specific profile or symmetry of the data are required. Vanishing topographies on the boundary are included in our analysis. Our method is inspired by recent results on the evolution of vortex rings in 3D axisymmetric incompressible fluids. \end{abstract}

\maketitle

\section{Introduction}

The purpose of this paper is to derive the point vortex dynamics for the lake equations, which are a two-dimensional geophysical model essentially described by {\sc H. P. Greenspan} \cite[p. 235]{Gr}.
Given a domain $\Omega\subset \R^2$ and a topography (depth function) $b:\overline{\Omega}\rightarrow \R_+$, the pair $(\Omega,b)$ is called a lake and the lake equations read
\begin{equation}\label{eq:lake}
 \begin{cases}
 \partial_t(bv)+\diver(bv\otimes v)+b\nabla p=0 \quad \text{on } \Omega\\
 \diver(bv)=0 \quad \text{on } \Omega, \quad (bv)\cdot \mb n=0 \quad \text{on } \partial\Omega,
 \end{cases}
\end{equation}
where $v:\Omega\rightarrow \R^2$ is the velocity field, $p$ the pressure and $\mb n$ the outward-pointing unit normal vector on $\partial\Omega$. 
The vector field $v$ corresponds to the vertically averaged horizontal component of a three-dimensional velocity, see \cite{BrGl} for a rigorous derivation of this system as the low Froude number limit of the usual inviscid shallow water equations. In the case where $b$ is constant, System~\eqref{eq:lake} reduces to the well-known 2D incompressible Euler equations, for which well-posedness is established in various settings in the works of {\sc W. Wolibner} \cite{Wolibner} and {\sc V. I. Yudovich} \cite{yudo}. As for the 2D Euler equations, the notion of vorticity plays a prevalent role in the qualitative analysis of the solutions of the lake equations, especially for vortex like solutions. More precisely, we consider the potential vorticity $\omega$ as
\begin{equation*}
 \omega(t,x):=\frac{\curl v(t,x)}{b(x)} = \frac{(\partial_{1}v_{2}-\partial_{2}v_{1})(t,x)}{b(x)}
\end{equation*}
which satisfies the continuity equation together with the incompressibility condition
\begin{equation*}
 \partial_t(b\omega)+\diver(b v \omega)=0,\quad \diver(bv)=0\quad \text{on } \Omega.
\end{equation*}
When the lake $(\Omega,b)$ has a finite number of islands $\Omega=\widetilde\Omega\setminus \Big(\bigcup_{k=1}^{N_{is}} \cal I_{k}\Big)$, see Assumption~\ref{assum:lake} below for the precise geometrical hypotheses, this amounts to the following vorticity formulation
\begin{equation}\label{Inviscid-vort}
\left\{
\begin{array}{l}
\partial_t(b\omega)+\diver(bv\omega) =0 \quad \text{on } \Omega \vspace*{0.2cm}\\
\curl v = b\omega, \quad \diver(bv)=0\quad \text{on } \Omega, 
\\ (bv)\cdot \mb n=0\quad \text{on } \partial\Omega,\\
\displaystyle\oint_{\partial\cal I_{k}} v\cdot \mb \tau \, \dd s=\Gamma^{k}\quad \text{for }k=1,\dots,N_{is}
\end{array}
\right.
\end{equation}
which is equivalent to \eqref{eq:lake} for smooth lakes, see \cite[Appendix A]{LNP}. The conservation of the circulation is related to Kelvin's theorem, see \cite[Proposition 2.13]{LNP}.
When the depth $b$ is bounded away from zero, existence and uniqueness of global weak solutions to \eqref{eq:lake} is proven by {\sc C. D. Levermore, M. Olivier} and {\sc E. S. Titi} in \cite{LOTiti2}.
In \cite{BM}, {\sc D. Bresch} and {\sc G. M\'{e}tivier} include varying depth functions vanishing on the boundary. 
Subsequently, {\sc C. Lacave, T. Nguyen} and {\sc B. Pausader} \cite{LNP} extend the work in \cite{BM} treating the case of singular domains and rough bottoms. They show that the inviscid lake equations are structurally stable under Hausdorff approximations of the fluid domain and $L^{p}$-perturbations of the depth. This study is extended by the present authors for an evanescent or emergent island in \cite{HLM} giving rise to an asymptotic dynamics including an additional Dirac mass in the vorticity: a point vortex located in $z^0\in \Omega$ with $b(z^0)=0$ and $\nabla b(z^0)=0$.
 Very recently, {\sc B. Al Taki} and {\sc C. Lacave} establish in \cite{AlTakiLacave} the existence and uniqueness of global classical solutions. Moreover, it is shown in \cite{AlTakiLacave} that these classical solutions are recovered in the vanishing viscosity limit for the viscous lake equations. 

In this paper, we aim to rigorously justify the $\eps$-limit of \eqref{Inviscid-vort} for $N$ vortices of core size proportional to $\eps>0$, namely for solutions to \eqref{Inviscid-vort} with initially sharply concentrated vorticity around $N$ distinct points. Our analysis applies to any fixed number $N$ of vortices without restrictions on the size or sign of their intensities, lakes $(\Omega,b)$ with vanishing topographies on the boundary are included. {\sc G. Richardson} \cite{richardson} formally computes the asymptotic motion law depending on the topography variations: to leading order, the vortex centers evolve according to the ODE
\begin{equation}\label{eq:Richardson}
 \dot{z}_{i}(t)=-|\ln\eps|\frac{\gamma_{i}}{4\pi}\frac{\nabla^{\perp}b(z_{i}(t))}{b(z_{i}(t))}, 
\end{equation}
where $\gamma_{i}=\int_{\Omega}b\omega_{i} \,\dd x$ and $\nabla^{\perp}b=(-\partial_2 b,\partial_1 b)$. So, vortices move along the level lines of the depth function $b$. For smooth lakes and depth functions bounded away from zero, the motion law \eqref{eq:Richardson} is rigorously proven by {\sc J. Dekeyser} and {\sc J. Van-Schaftingen} \cite{DekeyserVanS} for a single vortex. The method developed in \cite{DekeyserVanS} does not seem to allow for an extension of the theory to several point vortices and vanishing topography on the boundary. No strong localization property of the vorticity around its asymptotic trajectory is provided.

The strategy we pursue here is totally different from the one in \cite{DekeyserVanS}, that relies on sophisticated elliptic techniques. It is rather inspired by the analogy between point vortices for the lake equations and vortex rings for $3D$ axisymmetric inviscid fluids. Since the work by {\sc L. Da Rios} \cite{da-rios}, the dynamics of an isolated vortex filament in $3D$ inviscid fluids is conjectured to be well modelled by the binormal flow, but the rigorous derivation in the general case is an open problem, see \cite{JerrardSeis} for a result in this direction.
In the special case of (one or several) vortex rings, recent progress is due to {\sc P. Butt\`a, G. Cavallaro} and {\sc C. Marchioro} \cite{Mar3}, who rigorously prove that the leading order term in the velocity of vortex rings starting from sharply localized vortex rings with same axis of symmetry (say $\text{span}(\vec e_{z}$)), of radii $r_{i}$, with $r_{i}\neq r_{j}$ if $i\neq j$, and circulation $\gamma_i$, is $|\ln \varepsilon| \frac{\gamma_{i}}{4\pi}\frac{\vec e_{z}}{r_{i}}$. In the same spirit for \eqref{Inviscid-vort}, we consider $N$ vortices placed at mutual distance of order one and show that the velocity of the vortex centers is of order $\gamma_{i}|\ln \varepsilon|$ at leading order, which therefore is a rigorous justification of Richardson's law \eqref{eq:Richardson}. 
The leading order term in both problems stems from the 3D nature of the considered phenomena. After dimension reduction of the respective problems (by axisymmetry for the vortex rings \cite[Equation (1.6)]{Mar2} and vertically averaging of $v$ \cite{BrGl,Gr} for \eqref{eq:lake}), the leading order contribution is linked to the anelastic constraint for $v$ in \eqref{eq:lake}, see also \eqref{eq:decomposition} below. This contribution vanishes in the case of the 2D Euler equations, namely $b=\mathrm{const.}$, see \cite{MarPul93, Marchioro-Pulvirenti-book}. 

In order to highlight the main order contribution, we choose as in \cite{DekeyserVanS} to accelerate the time scale: we set $s=|\ln \varepsilon| t \in [0,T]$. Denoting by $\omega_{\varepsilon}$ the unique weak solution of \eqref{Inviscid-vort} with the initial data $\omega_{\varepsilon}^0$, we consider the rescaled vorticity
$$\tilde\omega_\varepsilon(s,x)=\omega_{\varepsilon}\left(\frac{s}{|\ln \eps|},x\right)$$ 
and $\tilde v_\eps$ the corresponding velocity field. In this rescaling, \eqref{Inviscid-vort} reads
\begin{equation}\label{eq:continuity}
\left\{
\begin{array}{l}
\partial_s(b \tilde\omega_\eps)+\frac{1}{\lne}\diver(b \tilde v_\eps\tilde \omega_\eps)=0 \quad \text{on } (0,\infty)\times \Omega \vspace*{0.2cm}\\
\curl \tilde v_{\varepsilon} = b\tilde\omega_{\varepsilon}, \quad \diver(b\tilde v_{\varepsilon})=0\quad \text{on }[0,\infty)\times \Omega, \\
 (b\tilde v_{\varepsilon})\cdot \mb n=0\quad \text{on } [0,\infty)\times\partial\Omega\\
\displaystyle\oint_{\partial\cal I_{k}} \tilde v_{\varepsilon}\cdot \mb \tau \, \dd s=\Gamma^{k}_\varepsilon \quad \text{for }k=1,\dots,N_{is} \text{ and on }[0,\infty)
,
\end{array}
\right.
\end{equation}
with the initial data $\tilde\omega_\eps(0,\cdot)=\omega_{\varepsilon}^0$.

\begin{remark}
We point out here that this time rescaling is equivalent to keeping the time invariant and to considering small vortices $ |\ln \varepsilon|\gamma_{\varepsilon} = \mathcal{O}(1)$ as in \cite{Mar3}, where $\gamma_{\varepsilon}:=\int_{\Omega}b\omega_{\varepsilon}$. This equivalence comes from the non-linearity of the lake equations \eqref{eq:lake}: if we consider $\gamma_{\varepsilon}=\tilde\gamma /|\ln \varepsilon|$, $\Gamma_{\varepsilon}^k=\tilde\Gamma^k /|\ln \varepsilon|$, with $\tilde\gamma$ and $(\Gamma^k)$ independent of $\varepsilon$, then we set $\tilde\omega_{\varepsilon} = |\ln \varepsilon| \omega_{\varepsilon}$ which verifies $\int b\tilde \omega_{\varepsilon}= \tilde\gamma$. The velocities recovered through the div-curl problem satisfies $\tilde v_{\varepsilon} = |\ln \varepsilon| v_{\varepsilon}$ and we obtain that $(\tilde\omega_{\varepsilon},\tilde v_{\varepsilon})$ is the solution of \eqref{eq:continuity}.
\end{remark}

As we work only with these rescaled variables, we omit the tilde in the following and return to $t$ for the time variable, so $(\omega_{\varepsilon},v_{\varepsilon})(t,x)$ corresponds to the solution of \eqref{eq:continuity} in the sequel.

\subsection*{Assumptions and main result}

The lakes under consideration in this article are smooth, with $N_{is}\in \N$ islands and their depth function may vanish on the boundary. We give below a precise definition.

\begin{assumption}\label{assum:lake}
The lake $(\Omega,b)$ satisfies the following conditions
\begin{enumerate}[(i)]
\item $\displaystyle \Omega:= \widetilde\Omega \setminus\Big(\bigcup_{k=1}^{N_{is}} \cal I_{k}\Big)$, where $\widetilde\Omega$, $\cal I_{k}$ are bounded simply-connected subsets of $\R^2$, $\widetilde\Omega$ is open, $\cal I_{k}$ are disjoint and compact subsets of $\widetilde \Omega$, and $\partial\Omega\in C^3$,
\item $b(x)=c(x)\varphi^{\alpha}(x)$ with $\alpha \geq 0$, $c,\varphi\in C^{3}(\overline{\Omega})$ and $c(x)\geq c_{0}>0$ on $\Omega$,
\item $\Omega=\{\varphi>0\}$ with $\nabla \varphi\neq0$ on $\partial\Omega$.
\end{enumerate}
\end{assumption}
These conditions enable us to consider the simplest case of a non-vanishing shore ($\alpha=0$) or the more realistic case of a vanishing topography on the boundary ($\alpha>0$). The function $\varphi$ plays the role of the distance to $\partial\Omega$ close to the boundary: see Lemma~\ref{lem:d-phi}.

For a vanishing topography on the boundary, these assumptions ensure that the level sets $\{x\in \overline{\Omega}\, : \ b(x)=b_{0}\}$ are unions of disjoint connected compact subsets of $\Omega$ for all $b_{0}>0$. The compactness in $\Omega$ of the level set is also true for a non-vanishing shore if $b_{0}\notin b(\partial\Omega)$. If $b_{0}\in b(\partial\Omega)$, we will consider a connected component of the respective level set being disjoint to the boundary, see the following remark.

\begin{remark}\label{rem:ODE}
Under Assumption \eqref{assum:lake} and given $z^0$ such that the connected component $\mathcal{C}$ of $\{x \in \overline{\Omega} : b(x)=b(z^0)\}$ containing $z^0$ satisfies $\mathcal{C}\cap \partial\Omega=\emptyset$, the Cauchy-Lipschitz theorem yields that \eqref{eq:Richardson} has a unique global solution in $C^1(\R^+)$ where $z(0)=z^0$. Even if $b$ may vanish at the boundary, the solution is indeed global because it holds $b(z(t))=b(z^0)$ for all $t\geq 0$. The assumption on the disjointness of $\mathcal{C}$ and $\partial\Omega$ only means that a point moving on this level set cannot reach the boundary.
\end{remark}
 
We have stated Assumption~\ref{assum:lake} with only one commun coefficient $\alpha$, but we can, as in \cite{LNP,HLM}, specify different rates $\alpha_{k}$ around $\cal I_{k}$. This extension is trivial because we only derive estimates in this paper in compact subsets of $\Omega$. However, in the case of different $\alpha_{k}$, we would need to split $\Omega$ into different neighborhoods, which makes the assumption more complicated to state. 

Next, we specify the assumptions on the initially concentrated vorticity $\weps^0$ around $N$ point vortices which belong to different and mutually disjoint connected components of level sets. 

\begin{assumption}\label{ass: initial data}
Let $N_{v}\in \N^*$ and $(z_{1}^0,\ldots,z_{N_{v}}^0)\in \Omega^{N_{v}}$, we define $\cal C_{i}$ to be the connected component of the level set $\{x\in \overline{\Omega} : b(x)=b(z_{i}^0)\}$ containing $z_{i}^0$.

The points $(z_{i}^0)$ and the sequence of initial vorticities $\omega_\eps^0$ satisfy:
\begin{enumerate}[(i)]
\item $ \displaystyle \cal C_{i}\cap \cal C_{j}= \cal C_{i}\cap \partial\Omega=\emptyset, \quad \text{ for all }i\neq j$.
\item There exist $M_{0},\varepsilon_{0}>0$ and $(\delta_{1},\ldots,\delta_{N_{v}})\in \{-1,1\}^{N_{v}}$ such that 
\begin{equation*}
 \weps^0(x):=\sum_{i=1}^{N_v}\omega_{\eps,i}^0(x),
\end{equation*}
where $\omega_{\eps,i}^0\in L^\infty(\Omega)$ satisfies the following properties for all $\varepsilon\in (0,\varepsilon_{0}]$ and $i=1, \dots, N_v$:
\begin{align*}
 \bullet\quad &\supp(\omega_{\eps,i}^0)\subset B(z_i^0,M_0\eps); \\
 \bullet\quad & 0\leq \delta_{i} \omega_{\eps,i}^0 \leq \frac{M_0}{\eps^2} ;\\
 \bullet\quad & \gamma_{\eps,i} := \int_{\Omega}b\omega_{\eps,i}^0 \,\dd x \to \gamma_{i} \in \R^* \text{ as }\varepsilon\to 0.
\end{align*}
\end{enumerate}
\end{assumption}

\begin{remark}\label{rem:r0}
As the $\mathcal{C}_i$ containing the vortex center $z_i^0$ are mutually disjoint connected compact sets, we state that there is $\rho_{1,2}>0$ such that for all $\rho \in (0,\rho_{1,2}]$ the compacts subsets $\cal C_{\rho,1}$ and $C_{\rho,2}$ are disjoint, where
 \[
 \cal C_{\rho,j}:=\text{the connected component containing $z_{j}^0$ of } \{x\in \overline{\Omega} : \: |b(x)-b(z_{j}^0)|\leq \rho \}.
 \]
 Indeed, by contradiction, we would have a decreasing sequence $\rho_{n}\to 0$ and $x_{n}\in \cal C_{\rho_{n},1}\cap C_{\rho_{n},2}$ which would imply by extracting a subsequence redenoted by $x_n$ that $x_{n}\to x_{\infty}$ such that $b(x_{\infty})=b(z_{1}^0)= b(z_{2}^0)$. This means that $\cal C_{\rho_{n},1}=\cal C_{\rho_n,2}$ is a decreasing sequence of connected compact sets, hence verifying that the limit is connected, contained in the level set $\{x\in \overline{\Omega} : \: b(x)=b(z_{1}^0) \}$ and containing $z_{1}^0$ and $z_{2}^0$, which is in contradiction with $\cal C_{1} \cap \cal C_{2}=\emptyset $. In the same way, observing that $\partial\widetilde\Omega$ and $\partial \cal I_{k}$ are connected compact sets, we also infer that $\cal C_{\rho,1}$ and $\partial\Omega$ are disjoint for $\rho\leq \rho_{01}$. Defining $\rho_{b}:=\min_{0\leq i<j\leq N_v} \rho_{i,j}$, we get that 
 $(\cal C_{\rho,i})_{i=1,\dots,N_v}$ and $\partial\Omega$ are mutually disjoint connected compact sets for all $\rho\leq \rho_{b}$.
Decreasing $\varepsilon_{0}$ if necessary, this allows us to introduce $r_{0}>0$ and $C>0$ such that for all $i,j\in \{1,\dots ,N_v\}$, $i\neq j$ 
 \begin{equation*}
 \dist(\mathcal{C}_{\rho_{b},i},\partial\Omega)>r_0, \quad \dist(\mathcal{C}_{\rho_{b},i}, \mathcal{C}_{\rho_{b},j})> r_0, \quad \supp(\omega_{\eps,i}^0)\subset \mathcal{C}_{C\eps,i}\subset
\mathcal{C}_{\rho_{b}/2,i}
\end{equation*}
for all $\varepsilon\in (0,\varepsilon_{0}]$. Here and in all the sequel, the distance of two compact subsets of $\R^2$ corresponds to their minimal distance (and not to the Hausdorff distance):
\[
\dist(K_{1},K_{2}):=\min_{x\in K_{1}}\dist(x,K_{2})=\min\{ |x-y|\, :\, (x,y)\in K_{1}\times K_{2}\}.
\]
This ensures in particular that the supports of $\omega_{\eps,i}^0$ are mutually disjoint compact subsets of $\Omega$.
\end{remark}

In our main result, we prove that, for concentrated initial data, the vorticity evolving through \eqref{eq:continuity} remains weakly concentrated in moving disks and strongly concentrated in the direction of the steepest ascent of the topography $b$.

\begin{theorem}\label{thm:main}
Let $(\Omega,b)$ be a lake satisfying Assumption~\ref{assum:lake}, let $\left(\omega_\eps^0, (z_k^0)_{k=1}^{N_v}\right)$ be a sequence of initial data satisfying Assumption~\ref{ass: initial data} and let $(\Gamma_{\varepsilon}^k)_{k=1}^{N_{is}}$ be a sequence in $\R^{N_{is}}$ converging to $(\Gamma^k)_{k=1}^{N_{is}}$ as $\varepsilon\to 0$.

 For all $i=1,\ldots,{N_v}$, let $z_i$ be the unique global solution of 
\begin{equation}\label{eq: zi}
 \dot{z_i}(t)=-\frac{\gamma_{i}}{4\pi} \frac{\nabla^{\perp}b(z_i(t))}{b(z_i(t))}, \quad z_i(0)=z_i^0.
\end{equation}

Then the following hold for the global weak solution\footnote{Given by the well-posedness results by {\sc D. Bresch} and {\sc G. M\'{e}tivier} in \cite{BM} as well as by the second author, {\sc T. Nguyen} and {\sc B. Pausader} in \cite{LNP}, see Proposition~\ref{prop:well-posed} below.} $\omega_{\varepsilon}$ of \eqref{eq:continuity} with initial data $\omega_{\varepsilon}^0$:

(i) There exist
trajectories $z_{\eps, i}\in C(\R_+; \Omega)$ such that 
\begin{equation*}
 \sup_{t\in [0,T]} | z_{\varepsilon,i}(t) - z_{i}(t) | \to 0\quad \text{as }\varepsilon\to 0 ,\quad \text{for all }T\geq 0\text{ and }1\leq i\leq {N_v}.
\end{equation*}
There exists a decomposition
\[
 \weps(t,x)=\sum_{i=1}^{N_v}\omega_{\eps,i}(t,x), \quad \omega_{\eps,i}\in L^\infty(\R_+,L^\infty(\Omega))\cap C(\R_+,L^p(\Omega))\quad \text{for }p\in [1,\infty)
\]
verifying $\omega_{\eps,i}(0,\cdot) =\omega_{\eps,i}^0$ and such that \begin{equation*}
 \sup_{t\in [0,T]}\Bigg| \gamma_{i} - \int_{B(z_{\varepsilon,i}(t) ,R_{\varepsilon})} b\omega_{\varepsilon,i}(t,x)\, \dd x \Bigg| \to 0 \text{ as }\varepsilon\to 0 , \text{ for all }T\geq 0\text{ and }1\leq i\leq {N_v} 
\end{equation*}
where $R_\eps \to 0$ as $\eps \to 0$.

(ii) For every $k\in (0,1/4)$ and $T>0$ there exists $\eps_{k,T}, C_{k,T}>0$ depending only on $(z_i^0)_{i=1,\dots, {N_v}}$, $(\gamma_i)_{i=1,\dots, {N_v}}$, $(\Gamma^i)_{i=1,\dots, {N_{is}}}$, $M_{0}$, $b$, $\Omega$, $k$ and $T$ such that, for any $\varepsilon\in (0,\varepsilon_{k,T}]$, 
$\supp \omega_{\eps,i}(t,\cdot) \subset \cal C_{\rho_{b},i}.$
Moreover,
\begin{equation}\label{Eq.strong}
 \supp \omega_{\varepsilon,i}(t,\cdot) \subset \left\{ x \, : \, |b(x)-b(z_i^0)| \leq \frac{C_{k,T}}{|\ln \varepsilon|^k}\right\}
 \end{equation}
for all $t\in[0,T]$ and $i\in \{1,\dots,{N_v}\}$.
\end{theorem}

We stress that the result allows one to consider any fixed number ${N_v}$ of vortices that are initially well-separated in the sense of Assumption~\ref{ass: initial data} and it also applies to lakes $(\Omega,b)$ with beaches, namely vanishing topographies at the boundary. Note that the asymptotic dynamics \eqref{eq: zi} is a consequence of the obtained localization properties which are pivotal to our method. The first statement corresponds to a \emph{weak concentration} property as obtained in \cite{DekeyserVanS} for a single vortex. While the main part of each blob of vorticity is included in a small disk close to the asymptotic vortex trajectory, some small amount of vorticity can be located outside this disk. The second item concerns the localization property for the blobs. A direct byproduct of the inclusion $\supp \omega_{\eps,i}(t,\cdot) \subset \cal C_{\rho_{b},i}$ is the fact that the blobs remain separated on $[0,T]$. It also states the \emph{strong localization} of $\omega_{\eps,i}$ around the connected component $\mathcal{C}_i$ of the level set containing $z_{i}^0$. More precisely, a small amount of vorticity can be located outside a small disk around the vortex center, but has to be close to $\cal C_{i}$ (see Remark~\ref{rem:dim}). This stronger notion of localization compared to the weak concentration is crucial for treating several vortices and vanishing topographies. Note that the authors in \cite{DekeyserVanS} lack such a strong localization property and hence only consider a single vortex. 

\begin{remark}\label{rem:dim}
 We emphasize that the strong localization \eqref{Eq.strong} in the direction of steepest ascent of $b$ gives the localization around $\cal C_{i}$ for the usual topology. Indeed, if we draw any closed curve $\Lambda$ which does not intersect $\cal C_{i}$, but which can be close to $\cal C_{i}$, then we can apply Remark~\ref{rem:r0} to state that there is $\rho_{\Lambda}>0$ such that $C_{\rho_{\Lambda},i}$ is disjoint to $\Lambda$, because $\Lambda$ is a connected compact set disjoint to $\cal C_{i}$. This implies that for all $\varepsilon$ such that $\frac{C_{k,T}}{|\ln \varepsilon|^k}<\rho_{\Lambda}$, the vorticity cannot meet the curve $\Lambda$ on $[0,T]$. For more details on this argument, see the end of the proof of Theorem~\ref{thm:main} in Section~\ref{sec:ProofMain}.
\end{remark}

The axisymmetric 3D Euler equations without swirl can be interpreted as a special case of the lake equations for which $\Omega=\R\times \R_{+}$ and $b(z,r)=r$, see for instance the introduction in \cite{Mar2}. As mentioned in the introduction, vortex rings in the $3D$ problem amount to point vortices for \eqref{eq:lake} and we recall that the limit motion in \cite{Mar2,Mar3} is 
\[
\dot{z}_{i}(t)=\frac{\gamma_{i}}{4\pi} \frac{\vec{e}_{z}}{r_{i}^0} = -\frac{\gamma_{i}}{4\pi} \frac{(0,1)^\perp}{r_i(t)}.
\]
 In this regard, Theorem~\ref{thm:main} can be seen as a generalization to general $\Omega$ and $b$ of the results developed in \cite{Mar3} and references therein. However, such a generalization is not trivial because \cite{Mar3} and previous results rely on the 3D explicit Biot-Savart law, which is not available for general $b$. One of the main difficulties is to provide a proof of Theorem~\ref{thm:main} without this kind of representation formulas. 

The assumption that the initial centers of vorticity $z_{i}^0$ belong to different connected components of level sets can be understood as a sufficient condition to rule out collisions between vortices or with the boundary $\partial\Omega$. This condition is reminiscent of requiring the vortex rings to be of different radii in \cite{Mar3}. The question whether a local result can be achieved if two vortices belong to the same $\cal C_{i}$ then arises naturally. Given that such a local result was proven for 3D vortex rings in \cite{Mar2}, we expect that a respective local result for the lake equations \eqref{eq:lake} is achievable with our approach, even with a simplified proof compared to Theorem~\ref{thm:main} as difficulties related to the global validity are then not relevant. More precisely, the local result would yield $T_{0}>0$ such that the limit dynamics is proven on $[0,T_{0}]$. As in \cite{Mar2} we expect $T_0$ to be strictly smaller than the time of first collision between the limit point vortices $T_c$ which can be $T_c=+\infty$. This restriction for the aforementioned models arising from $3D$ models after dimension reduction is a major difference to the 2D Euler case for which the result is known up to collision time \cite{Marchioro-Pulvirenti-book}. 
As the strong localization only holds in one direction, a collision of a vortex with a small filament emitted by the preceding vortex core can occur at time $T_{\varepsilon}<T_{c}$.
The spreading of the support is related to the phenomenon of filamentation of vortex patches which is well observed experimentally and numerically, see e.g. \cite{Leweke1,Leweke2}. Vortex filamentation is not yet fully understood and expected to be linked to various instability mechanisms, see e.g. \cite{Widnall}.

For the 2D Euler eq., {\sc C. Marchioro} and {\sc M. Pulvirenti} derive a lower bound for the time of the first appearance of vortex filamentation \cite{Marchioro-Pulvirenti-book}. Specifically, they show that there exists $\varepsilon_{T}$ depending on $T$ such that vorticity initially compactly supported in a disk of radius $\varepsilon$ remains supported in a disk of radius $\varepsilon^\beta$ on $[0,T]$, provided that $\varepsilon\in (0,\varepsilon_{T}]$. It remains a fascinating open question to fully understand the spreading of the support. An important step in that direction consists in determining the optimal $\beta>0$. While $\beta<1/300$ is assumed in \cite{Marchioro-Pulvirenti-book}, this bound is improved to $\beta<1/3$ in \cite{Mar98} by {\sc C. Marchioro}. In \cite{ButtaMar18}, {\sc P. Butt\`a} and {\sc C. Marchioro} investigate the time for which a small filament emanates from the disk of radius $\varepsilon^\beta$ with $\beta<1/2$. Understanding the respective mechanism for the dimension reduced 3D models appears to be even more involved from a mathematical point of view. Indeed, decomposing the velocity field as in \eqref{eq:decomposition} below, suggests that the first order term corresponds to the analogous one as in the 2D case which does not contribute to the displacement of the center of vorticity. The second order term in the expansion determines the motion at scale of velocity $\mathcal{O}(\lne)$, see \eqref{eq:uL} below. Being related to the anelastic constraint, its counterpart for 2D Euler vanishes. The special structure of this term pointing in the direction of $\vec{e}_{z}$ for vortex rings and $\nabla^\perp b$ for \eqref{eq:lake}, yields the motion of the vortex core with velocity of order $\lne$ in that direction, which corresponds to the binormal flow in this special case. Note that the velocity in the exterior of the vortex core is of $\mathcal{O}(1)$. For the lake equation, this significant difference of velocity heuristically suggests the mechanism leading to the spreading of the support: a small filament emanating from the disk of radius $\varepsilon$ is transported by a velocity which is much smaller than the one in the core of the vortex.
The initial vortex will spread in the direction $\nabla^{\perp} b$, we are only able to prove strong localization in the direction of steepest ascent of $b$, namely close to the level set. 
In order to avoid this spread, one would need to prove that the vorticities remain confined in a disk of radius $R\varepsilon$, which is far from being known even in dimension two.
The link of the localization problem to the vortex filamentation phenomena and the stability of vortex filaments further motivates its study. The vortex filamentation effect is of high physical relevance and it has been extensively investigated in the physical literature starting from the pioneering work \cite{Thomson} by {\sc J.J. Thomson}, see e.g. \cite{Widnall}. On a related note, the understanding of the vortex dynamics for \eqref{eq:lake} may be of practical use to understand phenomena such as rip currents, transport mechanisms of pollutants and sedimentation, see \cite{Peregrine} and also the introduction and conclusions in \cite{richardson} as well as references therein.

Similarly to the aforementioned results, Theorem~\ref{thm:main} states the stability of the motion of point vortices requiring only a very general localization assumption, see Assumption~\ref{ass: initial data}. The result is therefore sufficiently robust and general to allow for an experimental and numerical observation of the phenomenon. In particular, we do not assume any special profile or symmetry property for the initial data. 

Concerning possible generalizations of Theorem~\ref{thm:main} we mention that, thanks to the strong localization property around the connected components $\cal C_i$ of the respective level set of $b$, we require the regularity properties of the lake $(\Omega,b)$ only locally close to $\mathcal{C}_i$. Hence, we expect an extension for lagrangian solutions, see Proposition~\ref{prop:transport}, to rough lakes, as considered in \cite{LNP, HLM}, to be within reach by the present techniques.

Without commenting further on the dynamics of the limit trajectories, we only mention that no assumptions on the zero-set of $\nabla b$ are required. Note that under suitable regularity assumptions level sets can be characterized by means of Sard's Theorem \cite{Sard} and its generalizations.
We do not rely on any such property of the level sets. If $z_{i}^0$ is located on a critical point $\nabla b(z_{i}^0)=0$, then Theorem~\ref{thm:main} states that $\omega_{\varepsilon,i}$ remains concentrated around a stationary point. It is then an interesting open problem to investigate the motion in this case without rescaling the time, i.e. exhibiting the second term in the motion expansion. In other words, Theorem~\ref{thm:main} characterizes the dynamics on a short time scale of order $\lne^{-1}$. The mutual interaction of vortices is only relevant on time scales of order one. Alternatively, to observe the interactions at the main order, we could initially place the vortices at a distance of order $1/\lne$. In this case, our strong localization result does not suffice to separate them because we would need $k\geq1$ in \eqref{Eq.strong}.

The paper is organized as follows. Section~\ref{sec:pre} concerns the well-posedness of the lake equations \eqref{eq:lake}. In Section~\ref{sec:redsystem}, it is shown how the problem addressed in Theorem~\ref{thm:main} can be reduced to the study of a single vortex evolving according to the lake equations with additional external field. The proof of the main Theorem~\ref{thm:main} is completed assuming that the reduced problem is solved, namely that Theorem~\ref{thm: main red system} holds. The remaining part of the paper is dedicated to the proof of Theorem~\ref{thm: main red system}, see the last paragraph of Section~\ref{sec:redsystem} for a detailed outline of Sections~\ref{sec:BS} to \ref{sec:trajectory}. The Appendix~\ref{sec:transport-appendix} is devoted to the proof of several regularity properties of solutions to \eqref{Inviscid-vort} stated in Proposition~\ref{prop:transport}. Finally, Appendix~\ref{app:rearrangement} provides a result on the rearrangement of the mass used several times throughout the paper.

\bigskip
\noindent
{\bf Acknowledgements.} The authors warmly thank Romain Joly for interesting discussions on general geometrical assumptions related to connected components, see Remark~\ref{rem:r0}, and Emmanuel Russ for discussions concerning elliptic estimates with islands. This work is supported by the French National Research Agency in the framework of the project ``SINGFLOWS'' (ANR-18-CE40-0027-01). L.E.H. was funded by the Deutsche Forschungsgemeinschaft (DFG, German Research Foundation) - SFB 1283/2 2021 - 317210226.
E.M. acknowledges the project ``INFAMIE'' (ANR-15-CE40-01).

\section{Preliminary results: Well-posedness and Biot-Savart type law}\label{sec:pre}

\subsection{Well-posedness of weak solutions for smooth lakes and additional properties}\label{sec:WP}

We begin this section by recalling the well-posedness result for weak solutions proved by {\sc D. Bresch} and {\sc G. M\'{e}tivier} in \cite{BM}, see also \cite{LNP}.

\begin{proposition}[\cite{BM}]\label{prop:well-posed}Let $(\Omega,b)$ a smooth lake, i.e. verifying Assumption~\ref{assum:lake}, $(\Gamma^k)\in \R^{N_{is}}$ and $\omega^0\in L^\infty(\Omega)$. For all $T>0$ there exists a unique pair $(v,\omega)$ with $\omega\in L^\infty([0,T]\times \Omega)$ which is a solution of \eqref{Inviscid-vort} in the following sense:
\begin{enumerate}[(i)]
 \item $\sqrt b v \in L^\infty([0,T];L^2(\Omega))$ satisfies in the weak sense for a.e. $t\in [0,T]$
\begin{gather*}
 \diver(bv)=0 \text{ and } \curl v=b\omega \text{ in }\Omega,\ (bv)\cdot \mb n =0 \text{ on }\partial\Omega,\\
 \oint_{\partial\cal I_{k}} v\cdot \mb \tau \, \dd s=\Gamma^{k}\quad \text{for }k=1,\dots,N_{is};
\end{gather*}
 \item for all $\Phi\in C^1([0,T]\times \overline{\Omega})$, there holds for any $t\in [0,T]$
 \begin{align*}
 \int_{\Omega} \Phi(t,x) (b\omega)(t,x) \,\dd x &- \int_{\Omega} \Phi(0,x) (b\omega^0)(x) \,\dd x \\
 &= \int_{0}^t \int_{\Omega} ( b\omega)(s,x) \left(\partial_{t}\Phi + v\cdot \nabla \Phi\right)(s,x) \,\dd x\dd s.
\end{align*}
\end{enumerate}
Moreover, the solution satisfies the following additional regularity properties: $v\in L^\infty([0,T];C(\overline{\Omega}))$, $v\cdot \mb n=0$ on $\partial\Omega$ and
\[
\omega\in C([0,T];L^r(\Omega)), \quad v\in C([0,T];W^{1,r}(\Omega))
\]
for any $r\in [1,\infty)$.
\end{proposition}

Considering $\Phi\equiv 1$ it is then obvious that the vorticity mass is conserved:
\[
\int_{\Omega}b(x)\omega(t,x)\, \dd x=\int_{\Omega}(b\omega^0)(x)\, \dd x\quad \text{for all }t\in [0,T].
\]
 It is moreover proved in \cite[Theorem 2.1]{AlTakiLacave} that\footnote{Even if this result is stated in \cite{AlTakiLacave} for simply connected domains, it holds true for lakes with islands. Indeed, elliptic regularity in the interior of the lake comes from an easy local elliptic argument, whereas the regularity close to a boundary is exactly the difficult and established result in \cite{AlTakiLacave}. For more details about the localization of elliptic estimates around each boundary, see \cite[Sections 2.2 and 3]{DekeyserVanS}.} $v$ is log-lipschitz on $\overline{\Omega}$, namely 
 \[
 |v(t,x)-v(t,y)|\leq C(\|\omega\|_{L^\infty([0,T]\times \Omega)})|x-y|\big(1+\big|\ln |x-y|\big|\big).
 \]

We establish further properties of this solution.

\begin{proposition}\label{prop:transport}
Let $\omega^0\in L^\infty(\Omega)$ with compact support in $\Omega$, and set
$$\delta_0:=\dist(\supp \omega_{0},\partial\Omega)=
\inf\left\{ \mathrm{dist}(x,\partial\Omega):\:\: x\in \supp \omega^0\right\}.$$
Let $T>0$ and let $(v,\omega)$ be the unique weak solution of the lake equations on $[0,T]$ given by Proposition~\ref{prop:well-posed}. 
\begin{enumerate}
\item[(i)] There exists a compact subset $K_T$ of $\Omega$, depending only on $\|\omega^0\|_{L^\infty}$, $\delta_{0}$ and $T$, such that 
\[
\supp \omega(t,\cdot)\subset K_T, \quad \forall t\in [0,T].
\]
\item[(ii)] We have $\|b^{1/p}\omega(t,\cdot)\|_{L^p}=\|b^{1/p}\omega^0\|_{L^p}$, for $t\in [0,T]$, for $p\in [1,\infty)$.
\item[(iii)] There exists a unique flow associated to $v$ in the classical sense: for any $x \in \Omega$ and $t_0\in [0,T]$, there exists a unique characteristic curve $X(\cdot, t_0,x)\in C^1([0,T] ; \Omega)$ solving
\[
\frac{\dd X(t,t_0,x)}{\dd t}=v(t,X(t,t_0,x)), \quad X(t_0,t_0,x)=x.
\]
Moreover, we have $\omega(t,x)=\omega^0(X(0,t,x)))$ for all $t\in [0,T]$ and $x\in \Omega$.
\end{enumerate}
\end{proposition}
For simplicity we will further denote $X(t,x)=X(t,0,x)$.

We mention that stability and existence of renormalized solutions are established by {\sc D. Bresch} and {\sc P.-E. Jabin} \cite{BJ} for a class of advective equations with a vector field satisfying a degenerate anelastic constraint, which includes the case of the lake equations \eqref{Inviscid-vort} considered in the setting of our paper.

Proposition~\ref{prop:transport} is proved in Appendix~\ref{sec:transport-appendix}. The proof relies on several arguments from the theory of linear transport equations developed by {\sc R. J. Di Perna} and {\sc P. L. Lions} \cite{Dip} with vector fields with bounded divergence. The main difficulty here is that $\diver(v)=-\nabla b\cdot v/b$ is not bounded on $\Omega$ if $b$ vanishes on the boundary. 


Here, the fact that the log-lipschitz constant of the velocity field $v_\eps$ depends on $\|\omega_{\varepsilon}(t,\cdot)\|_{L^\infty}=\|\omega_{\varepsilon}^0\|_{L^\infty}$, diverging possibly as $\varepsilon^{-2}$, constitutes a major difficulty. Indeed, this does a priori not suffice to infer a lower bound for the distance to the boundary $\delta_{T,\varepsilon}$ when $\varepsilon$ tends to zero. Such a uniform control of the distance to the boundary will be included in the forthcoming definition of $T_{\varepsilon}$, see \eqref{eq:Teps main}, and one of the main consequence of the strong localization will be to state that $T_{\varepsilon}=T$.

\begin{remark}\label{rem:omega_i}
 In view of the last item of Proposition~\ref{prop:transport}, it is natural to define the decomposition of $\omega_{\varepsilon}$ in Theorem~\ref{thm:main} as the transport of the decomposition of the initial data, see Assumption~\ref{ass: initial data}: 
 \[\omega_{\varepsilon,i}(t,x):=\omega^0_{\varepsilon,i}(X_{\varepsilon}(0,t,x))).\]
\end{remark}

\subsection{Green kernel for the lake equations}

A classical feature of inviscid flow is the reconstruction of the velocity in terms of the vorticity. The vector field $v_{\varepsilon}$ is uniquely determined by $\omega_{\varepsilon}$ and $(\Gamma_{\varepsilon}^k)_{k=1}^{N_{is}}$ through the following div-curl problem
\begin{equation}\label{eq:divcurl}
\begin{aligned}
& \diver(b\veps)=0 \text{ in } \quad \Omega, \quad \curl\veps=b\weps \text{ in } \Omega, \quad b\veps\cdot \mb n =0 \text{ on }\partial\Omega,\\
 & \oint_{\partial\cal I_{k}} \veps\cdot \mb \tau \, \dd s=\Gamma^{k}_{\varepsilon}\quad \text{for }k=1,\dots,N_{is}.
\end{aligned}
\end{equation}
This vector field can be represented in terms of stream functions \cite[Proposition 2.10]{LNP},
\begin{equation}\label{BS:is}
 bv_{\varepsilon}=\nabla^{\perp}\Psi_{\varepsilon} + \sum_{k=1}^{N_{is}} \Big(\Gamma_{\varepsilon}^k + \int_{\Omega}b\omega_{\varepsilon}\varphi^k \Big)\nabla^\perp \psi^k,
\end{equation}
where $\Psi_{\varepsilon}\in H^1_0(\Omega)$ is the unique solution to 
\begin{equation}\label{eq:LapDir}
 \diver\left(\frac{1}{b}\nabla\Psi_{\varepsilon}\right)=b\omega_{\varepsilon} \quad \text{in} \quad \Omega.
\end{equation}
The functions $\varphi^k$ and $\psi^k$ are $b$-harmonic functions and form in particular a basis for the space of $b$-harmonic functions, we refer to \cite[Section 2.1]{LNP} for definitions. However, as they are independent of $\omega_{\varepsilon}$ and $\varepsilon$, we only need very weak properties for these functions, namely that
\begin{equation}\label{est:phipsi}
\varphi^k, \nabla \psi^k, D^2\psi^k\in L^\infty\Big(\bigcup_{i=1}^{N_{v}}\cal C_{\rho_{b},i}\Big) 
\end{equation}
where $\cal C_{\rho_b,i}$ is the neighborhood of $\cal C_{i}$ as defined in Remark~\ref{rem:r0}. This property comes from local elliptic regularity.

While in general no explicit Biot-Savart formula seems available for \eqref{eq:LapDir}, {\sc J. Dekeyser} and {\sc J. Van-Schaftingen} establish in \cite[Proposition 3.1]{DekeyserVanS} an interesting relation between the kernel $G_{\Omega,b}$ associated to this problem and the usual Laplacian kernel in bounded domain $G_{\Omega}$: we have 
\[
 G_{\Omega,b}(x,y) = G_{\Omega}(x,y)\sqrt{b(x)b(y)} + S_{\Omega,b}(x,y),
\]
where the remainder term is defined for all $y\in \Omega$ as a function $x\mapsto S_{\Omega,b}(x,y)$ such that 
\begin{equation}\label{eq:Sb}
\left\{\begin{aligned}
& \diver_{x}\Big( \frac1{b(x)} \nabla_{x} S_{\Omega,b}(x,y)\Big)= G_{\Omega}(x,y) \sqrt{b(y)} \Delta\frac1{\sqrt{b(x)}} \text{ in }\mathcal{D}'(\Omega), \\
 & S_{\Omega,b}(x,y)=0 \text{ for all }x\in \partial\Omega.
\end{aligned}\right.
\end{equation}
 For the non-vanishing topography, it is proved in \cite{DekeyserVanS} that $S_{\Omega,b}(\cdot,y)$ is in $W^{1,\infty}(\Omega)$. 
 In the present case of a vanishing topography, the existence and regularity of $S_{\Omega,b}(\cdot,y)$ is far from being obvious because $ \Delta\frac1{\sqrt{b(x)}}$ lacks to be integrable. This question is solved in \cite[Lemma 3.1]{AlTakiLacave} where the existence of a unique solution $S_{\Omega,b}(\cdot,y)$ is proven. Moreover, it was also established therein that for any $\delta>0$, there exists $C_{\delta}>0$ which depends only on $\Omega,b$ and $\delta$ such that
 \[
 \Big\| \frac1{\sqrt{b}} \nabla_{x}S_{\Omega,b}(\cdot,y) \Big\|_{L^2(\Omega)} \leq C_{\delta}, \quad \text{for all $y\in \Omega$ such that }{\rm dist}(y,\partial\Omega)\geq \delta.
 \]
 By the Poincar\'e inequality, it also follows that $ \Big\| S_{\Omega,b}(\cdot,y) \Big\|_{L^2(\Omega)} \leq C_{\delta}$.
Let 
\begin{equation}\label{eq:Omega delta}
\Omega_{\delta}:=\{x\in\Omega, \dist(x,\partial\Omega)\geq \delta\}.
\end{equation}
 As $b(x)\geq \tilde C_{\delta}>0$ on $\Omega_{\delta/2}$, the elliptic problem $\diver(b^{-1} \nabla\cdot )$ is non degenerate, and standard elliptic estimates give for any $p>2$
\[
 \| S_{\Omega,b}(\cdot,y) \|_{W^{2,p}(\Omega_{\delta})} \leq C_{\delta,p} \Big( \|G_{\Omega}(\cdot,y)\|_{L^{p}(\Omega_{\delta/2})} +\| S_{\Omega,b}(\cdot,y) \|_{L^{2}(\Omega_{\delta/2})} \Big) \leq C_{\delta}
 \]
which implies that
\begin{equation*}
 | S_{\Omega,b}(x,y) | + | \nabla_{x} S_{\Omega,b}(x,y) | \leq C_{\delta}, \quad \text{for all } x,y\in \Omega_{\delta}.
\end{equation*}
 
Finally, this leads to the stream function
\begin{equation}\label{eq:BS}
 \Psi_{\varepsilon}(x)=\int_\Omega G_{\Omega,b}(x,y) (b\omega_{\varepsilon})(y) \,\dd y.
\end{equation}

In the present paper, we are only interested in vorticities with support in $ \Omega_{r_{0}}$. Therefore, we can decompose $G_{\Omega,b}$ as follows: let
\begin{equation}\label{eq.Gdecomp}
 G_{\Omega,b}(x,y) = \frac1{2\pi}\sqrt{b(x)b(y)} \ln |x-y| + R_{\Omega,b}(x,y),
\end{equation}
 where the remainder part satisfies
\begin{equation}\label{est:R}
 | R_{\Omega,b}(x,y) | + | \nabla_{x} R_{\Omega,b}(x,y) | \leq C_{\delta}, \quad \text{for all } x,y\in \Omega_{\delta}.
\end{equation}
The estimates for $R_{\Omega,b}$ are derived from the estimates of $S_{\Omega,b}$ and from the fact that for all $y\in \Omega_{\delta}$ fixed, $\tilde R_{y}(x):=G_{\Omega}(x,y)-\frac1{2\pi}\ln|x-y|$ is harmonic in $\Omega$ and verifies $\tilde R_{y}(x)=-\frac1{2\pi}\ln|x-y|$ on $\partial\Omega$, which is bounded independently of $y\in\Omega_{\delta}$. Finally, we also use that $\nabla \sqrt{b}$ is bounded in $\Omega_\delta$.

\medskip

We conclude this section by recalling the mean value theorem for general $\Omega$ not necessarily convex. By the regularity of $\partial\widetilde\Omega$, where $\widetilde{\Omega}$ is defined in Assumption~\ref{assum:lake}, we know that $\widetilde\Omega_{\delta}$ is of the same shape as $\widetilde\Omega$ for $\delta>0$ small enough, namely a simply connected open bounded set. Similarly, for $\Omega_{\delta}$ one retrieves $\widetilde\Omega_\delta$ minus $N_{is}$ connected compact subsets. For such a $\delta$, we consider $\widehat{\Omega}_{\delta}$ such that $\Omega_{\delta} \subset \widehat{\Omega}_{\delta} \subset \Omega_{\delta/2}$ with $\partial\widehat{\Omega}_{\delta}$ being composed by $N_{is}+1$ Jordan curves. We claim that $\widehat{\Omega}_{\delta}$ is $a_{\delta}$-{\it quasiconvex} for some $a_{\delta}\geq 1$, that is, for any $x,y\in\widehat{\Omega}_{\delta}$ there exists a rectifiable path $\gamma\subset \widehat{\Omega}_{\delta}$ joining $x,y$ and satisfying $\ell(\gamma) \leq a_{\delta} |x-y|$.
This follows from the fact that $\widehat{\Omega}_{\delta} = \widetilde{\Omega}_{\delta}\setminus \bigcup_{k=1}^{N_{is}} \cal I_{\delta}$ where $\partial\widetilde{\Omega}_{\delta}$ and $\partial\cal I_{\delta}$ are piecewise $C^1$ Jordan curves with no cusp and hence a quasidisc (see, e.g., \cite{Gustafsson}). {\sc L. V. Ahlfors} shows in \cite{Ahlfors} that in 2D, we have
\begin{gather*}
 \partial\widetilde{\Omega}_{\delta} \text{ is a quasidisk} \Longleftrightarrow \widetilde{\Omega}_{\delta} \text{ is quasiconvex},\\
\partial \cal I_{k} \text{ is a quasidisk} \Longleftrightarrow\cal I_{k}^c\text{ is quasiconvex}.
\end{gather*}
Splitting $\widehat{\Omega}_{\delta}$ on neighborhoods of the boundary, this allows us to conclude that $\widehat{\Omega}_{\delta}$ is quasiconvex.

Therefore, for any $x,y\in \Omega_{\delta}$ and any $f\in C^1(\Omega_{\delta/2})$ we have
\begin{equation}\label{MVT}
 |f(x)-f(y)| = \Big|\int_{0}^1 \frac{\dd}{\dd s} (f(\gamma(s)))\,\dd s\Big| \leq \| \nabla f \|_{L^\infty(\Omega_{\delta/2})}a_{\delta} |x-y|.
\end{equation}

We note here that for a convex domain, we consider $\gamma$ the segment between $x$ and $y$, so $a_{\delta}=1$ in \eqref{MVT}. We will use \eqref{MVT} with $b^{-1}\in C^1(\Omega_{\delta/2})$ in view of Assumption~\ref{assum:lake}.

\section{The reduced system for a single vortex}\label{sec:redsystem}

As will be proved in Theorem~\ref{thm:main}, the asymptotic vortex dynamics is - at leading order - determined by the interaction with the topography $b$. We therefore aim to reduce the problem of proving the vortex dynamics to a problem for a single vortex, where the interaction between vortices is accounted for by an additional external field $F_\eps$:
\begin{equation}\label{eq:transport-1}
\left\{
\begin{array}{l}
\displaystyle \partial_t(b \omega_\eps)+\frac{1}{\lne}\diver(b (v_\eps+F_{\eps}) \omega_\eps)=0 \quad \text{in } (0,\infty)\times \Omega \vspace*{0.2cm}\\
\displaystyle\curl v_{\varepsilon} = b\omega_{\varepsilon} \quad \text{and}\quad \diver(b v_{\varepsilon})=0 \quad \text{in } [0,\infty)\times \Omega\\
 \displaystyle\oint_{\partial\cal I_{k}} \veps\cdot \mb \tau \, \dd s=\Gamma^{k}_{\varepsilon}\quad \text{on } [0,\infty) \text{ and for }k=1,\dots,N_{is} \\
\displaystyle (b v_{\varepsilon})\cdot \mathbf{n}=0\quad \text{on } [0,\infty)\times \partial\Omega, \quad
\omega_{\eps}(0,\cdot)=\omega_{\eps}^0 \quad \text{in } \Omega.
\end{array}
\right.
\end{equation}

The assumption concerning the initial data is then a trivial reduction of Assumption~\ref{ass: initial data}.

\begin{assumption}\label{data red}
Let $z^0\in \Omega$ and $\omega_\eps^0\in L^\infty(\Omega)$. We denote by $\mathcal{C}$ the connected component of the level set $\{x\in \overline{\Omega}\ :\ b(x)=b(z^0)\}$ containing $z^0$. Assume that $\cal C\cap \partial\Omega=\emptyset$ and that there exist $M_{0},\varepsilon_{0}>0$ and $\delta \in \{-1,1\}$ such that for all $\eps\in (0,\varepsilon_{0}]$ it holds
\begin{equation*}
 \supp(\weps^0)\subset B(z^0,M_0\eps), \quad 0\leq \delta\weps^0 \leq \frac{M_0}{\eps^2} , \quad \gamma_\eps:=\int_{\Omega}b\weps^0\dd x\Tend{\varepsilon}{0} \gamma\in \R^*.
\end{equation*}
\end{assumption}

We next need to define suitable assumptions on $F_\eps$, which are verified by the velocity field generated by the other vortices and which will be sufficient to prove our main reduced theorem. Namely, we consider that the external field $F_\eps$ is characterized by the following.
We consider the neighborhood $\cal C_{\rho_b}$ of $\cal C$ as defined in Remark~\ref{rem:r0} (we omit the index $i$ in the case of one vortex).
\begin{assumption}\label{ass: F}
The external field $F_{\eps}$ satisfies
\begin{enumerate}[(i)]
 \item the incompressibility and tangency conditions
 \begin{equation*}
 \diver(bF_{\eps})=0, \quad (bF_{\eps})\cdot \mb n=0,
 \end{equation*}
 \item a Lipschitz regularity estimate: for any $T>0$, there exists $C_{F},L_F>0$ depending only on $b$, $\Omega$, $(z_i^0)_{i=1,\dots, {N_v}}$, $(\gamma_i)_{i=1,\dots, {N_v}}$, $(\Gamma^i)_{i=1,\dots, {N_{is}}}$, $M_{0}$, $\varepsilon_{0}$, but not depending on $\eps\in (0,\varepsilon_{0}]$, such that for all $(t,x,y)\in [0,T]\times C_{\rho_b}\times C_{\rho_b}$ there holds 
 \begin{equation*}
 |F_{\eps}(x,t)|\leq C_{F}, \quad |F_{\eps}(t,x)-F_{\eps}(t,y)|\leq L_F|x-y|.
 \end{equation*}
\end{enumerate}
\end{assumption}
Concerning the flexibility of this approach for adaptions to different scaling regimes, we refer to Remark~\ref{rem:LipF}.

The reduced version of Theorem~\ref{thm:main} reads as follows.

\begin{theorem}\label{thm: main red system}
Let $(\Omega,b)$ be a lake satisfying Assumption~\ref{assum:lake}, the initial $\weps^0$ sharply concentrated in the sense of Assumption~\ref{data red}, $F_{\eps}$ satisfying Assumption~\ref{ass: F} and let $(\Gamma_{\varepsilon}^k)_{k=1}^{N_{is}}$ be a sequence in $\R^{N_{is}}$ converging to $(\Gamma^k)_{k=1}^{N_{is}}$ as $\varepsilon\to 0$. Let
\begin{equation}\label{eq: single z}
 \dot{z}(t)=-\frac{\gamma}{4\pi} \frac{\nabla^{\perp}b(z(t))}{b(z(t))}, \quad z(0)=z^0.
\end{equation}
Then the following holds for all $T>0$ and $\weps$ solution\footnote{Such that $\omega_{\varepsilon}\in C([0,T];L^r(\Omega))$ for any $r\in[1,\infty)$ and where $(V_{\varepsilon},\omega_{\varepsilon})$ is a weak solution in the sense of (ii) in Proposition~\ref{prop:well-posed} and satisfies Proposition~\ref{prop:transport} for $V_{\varepsilon}=\frac1{\lne}(v_{\varepsilon}+F_{\varepsilon})$.} to \eqref{eq:transport-1} with initial data $\weps^0$:
\begin{enumerate}[(i)]
\item The vorticity $\weps$ is strongly localized in the direction of steepest ascent of $b$, namely for every $k\in (0,1/4)$ there exist $\eps_{k,T}, C_{k,T}>0$ depending only on $(z_i^0)_{i=1,\dots, {N_v}}$, $(\gamma_i)_{i=1,\dots, {N_v}}$, $(\Gamma^i)_{i=1,\dots, {N_{is}}}$, $M_{0}$, $b$, $\Omega$, $k$ and $T$ such that $\supp \omega_{\eps}(t,\cdot) \subset \cal C_{\rho_{b}}$ for all $t\in [0,T]$. 
Moroever,
\begin{equation*}
 \supp \weps(t,\cdot)\subset \left\{x\in \Omega \, \, : \, |b(x)-b(z^0)|\leq \frac{C_{k,T}}{|\ln\eps|^k}\right\}
\end{equation*}
for all $\varepsilon\in (0,\varepsilon_{k,T}]$.
 \item The vorticity $\weps$ is weakly localized, namely there exist $z_{\eps}\in C([0,T];\Omega)$ and $C$ depending only on $(z_i^0)_{i=1,\dots, {N_v}}$, $(\gamma_i)_{i=1,\dots, {N_v}}$, $(\Gamma^i)_{i=1,\dots, {N_{is}}}$, $M_{0}$, $b$, $\Omega$ and $T$, such that for all $\eps\in (0,\varepsilon_{\frac18,T}]$ and $t\in [0,T]$
 \[
 \Bigg| \gamma - \int_{B(z_{\varepsilon}(t) ,R_{\varepsilon})} b\weps(t,x)\, \dd x \Bigg| \leq \frac{C}{\ln|\ln\varepsilon|}\quad \text{where }R_{\varepsilon}=\left(\frac{\ln|\ln\varepsilon|}{|\ln\varepsilon|}\right)^{1/2}.
 \]
 \item We have
 \[
 \sup_{t\in [0,T]} | z_{\varepsilon}(t) - z(t) | \to 0\quad \text{as }\varepsilon\to 0,
 \]
 where the limit trajectory $z$ is the solution of \eqref{eq: single z}.
\end{enumerate}
\end{theorem}

The proof of Theorem~\ref{thm: main red system} constitutes the main and difficult step towards the main result of this paper. Indeed, Theorem~\ref{thm:main} will then follow from Theorem~\ref{thm: main red system} upon proving that the vorticity $\weps$ is given by the superposition of vortices $\omega_{\varepsilon,i}$ that evolve according to \eqref{eq:transport-1} and where the external field accounting for the interaction with the other vortices satisfies Assumption~\ref{ass: F} for all times. For that purpose, the strong localization property in Theorem~\ref{thm: main red system} is crucial. The authors in \cite{DekeyserVanS} lack such a property and only consider a single vortex.

The strategy of reducing the problem to the evolution of a single vortex is inspired by the approach for point vortices \cite{Marchioro-Pulvirenti-book} and vortex rings \cite{Mar3}, see also references therein. However, for the lake equations \eqref{eq:lake} new difficulties arise due to the generality of both the geometry $\Omega$ and the topography $b$. For instance, no explicit Biot-Savart law is available - in contrast to $b=\mathrm{const.}$ ($2D$-Euler) and $b(z,r)=r$ (axisymmetric $3D$-Euler without swirl). Note that the aforementioned results for the respective problems build upon that explicit formula. Moreover, we consider general bounded domains and vanishing topographies that were not included in \cite{DekeyserVanS}. 

\subsection{Proof of Theorem~\ref{thm:main} assuming that Theorem~\ref{thm: main red system} is proved}\label{sec:ProofMain}

Under the assumptions of Theorem~\ref{thm:main}, let $(\veps,\weps)$ be the unique global solution to \eqref{eq:continuity} with initial data $(\veps^0,\weps^0)$ and where the velocity fields $\veps, \veps^0$ are uniquely determined by $\weps, \weps^0$ respectively through the div-curl problem \eqref{eq:divcurl}. This solution is given by Proposition~\ref{prop:well-posed} and satisfies Proposition~\ref{prop:transport}. In particular, we define $\omega_{\varepsilon,i}$ through the characteristics by Remark~\ref{rem:omega_i}, which is a solution of
\begin{equation}\label{eq:transport red}
\partial_t(b \omega_{\eps,i})+\frac{1}{\lne}\diver(b (v_{\eps,i}+F_{\eps,i}) \omega_{\eps,i})=0, \quad \omega_{\eps,i}(0)=\omega_{\eps,i}^0,
\end{equation}
where $bv_{\eps,i}=\nabla^\perp \Psi_{\varepsilon,i}+ \sum_{k=1}^{N_{is}} \Big(\Gamma_{\varepsilon}^k + \int_{\Omega}b\omega_{\varepsilon,i}\varphi^k \Big)\nabla^\perp \psi^k$ is uniquely determined by $\omega_{\eps,i}$ through \eqref{eq:BS}, and 
\begin{equation}\label{eq:Feps}
 F_{\eps,i}:=\frac1{b}\sum_{j\neq i}\Bigg( \nabla^\perp \Psi_{\varepsilon,j}+ \sum_{k=1}^{N_{is}} \Big( \int_{\Omega}b\omega_{\varepsilon,j}\varphi^k \Big)\nabla^\perp \psi^k \Bigg).
\end{equation}
In the previous definition $\Psi_{\varepsilon,j}$ is recovered from $\omega_{\eps,j}$ through \eqref{eq:BS}.

The only point to check in order to use Theorem~\ref{thm: main red system} is that $F_{\eps,i}$ satisfies Assumption~\ref{ass: F}, in particular {\it (ii)}, because {\it (i)} is already verified by the definition of $F_{\eps,i}$. For this purpose, we use the neighborhoods $\cal C_{\rho_{b},i}$ of $\mathcal{C}_i$ introduced in Remark~\ref{rem:r0}, with a distance $r_{0}$ separating to $\cal C_{\rho_{b},j}$ and $\partial\Omega$.

For $T>0$ be fixed, we set
\begin{multline}\label{eq:Teps main}
 T_{\eps}:=\sup\Big\{ t\in[0,T] \, \, : \, \supp \omega_{\eps,i}(s,\cdot) \subset \cal C_{\rho_{b},i} \\ \text{for all} \, s\in[0,t], \, i\in \{1,\dots,{N_v} \}\Big\},
\end{multline}
which implies that the vortex blobs are separated on $[0,T_{\varepsilon}]$. Note that by assumption on $\omega_{\eps,i}^0$, $T_\eps$ exists for $\eps$ sufficiently small depending only on $\rho_b$ fixed by Remark~\ref{rem:r0}. 

First, we show that $F_{\eps,i}$ defined in \eqref{eq:Feps} satisfies Assumption~\ref{ass: F}.\begin{lemma}\label{lem: bound F}
Under the assumptions of Theorem~\ref{thm:main}, for all $i\in \{1,\dots {N_v}\}$, $\cal C_{\rho_b,i}$ and $T_{\eps}$, as in Remark~\ref{rem:r0} and \eqref{eq:Teps main} respectively, there exists $C_{\Omega,b,r_0}>0$ independent of $\eps$ (depending only on $\Omega, b$ and $r_0$) such that
\begin{equation*}
|\nabla \Psi_{\varepsilon,j}(t,x)|\leq C_{\Omega,b,r_0} |\gamma_{\varepsilon,j}|, \quad |\nabla \Psi_{\varepsilon,j}(t,x)-\nabla \Psi_{\varepsilon,j}(t,y)|\leq C_{\Omega,b,r_0}|\gamma_{\varepsilon,j}| |x-y|,
\end{equation*}
for all $(t,x,y)\in [0,T_\eps]\times \cal C_{\rho_{b},i}\times \cal C_{\rho_{b},i}$ and $j\neq i$.
\end{lemma}
By \eqref{est:phipsi}, this lemma directly implies that the vector field $F_{\eps,i}$ as defined in \eqref{eq:Feps} satisfies
\begin{equation*}
 |F_{\eps,i}(t,x)|\leq C_F, \quad \left|F_{\eps,i}(t,x)-F_{\eps,i}(t,y)\right|\leq L_{F} |x-y|
\end{equation*}
for all $(t,x,y)\in [0,T_{\varepsilon}]\times \cal C_{\rho_{b},i}\times \cal C_{\rho_{b},i}$, where $C_{F}$ and $L_{F}$ depends only on $\Omega$, $b$, $M_{0}$, $(z_i^0)_{i=1,\dots, {N_v}}$, $(\gamma_i)_{i=1,\dots, {N_v}}$ and $(\Gamma_i)_{i=1,\dots, {N_{is}}}$, and therefore it satisfies Assumption~\ref{ass: F} on $[0,T_{\varepsilon}]$.

\begin{proof}[Proof of Lemma~\ref{lem: bound F}]
We recall from \eqref{eq:BS} and \eqref{eq.Gdecomp} that
\begin{align*}
\nabla \Psi_{\varepsilon,j}(t,x)
=&\nabla \int_{\Omega}\left(\frac1{2\pi}\sqrt{b(x)b(y)} \ln |x-y| + R_{\Omega,b}(x,y)\right)(b\omega_{\eps,j})(y)\, \dd y\\
=&\frac1{2\pi}\int_{\Omega}\frac{x-y}{|x-y|^2}\sqrt{b(x)b(y)}(b\omega_{\eps,j})(y)\, \dd y\\
&+\frac{\nabla b(x)}{4\pi b(x)}\int_{\Omega}\ln|x-y|\sqrt{b(x)b(y)}(b\omega_{\eps,j})(y)\, \dd y\\
&+\int_{\Omega}\nabla_x R_{\Omega,b}(x,y)(b\omega_{\eps,j})(y)\, \dd y=:I_1+I_2+I_3.\\
\end{align*}
 Note by Assumption~\ref{assum:lake} that $b\in W^{1,\infty}(\cal C_{\rho_{b},i}\cup \cal C_{\rho_{b},j} )$, and by Remark~\ref{rem:r0} that $|x-y|\geq r_{0}$ for all $(x,y)\in \cal C_{\rho_{b},i}\times \cal C_{\rho_{b},j}$ which allows us to estimate $I_{1}$:
\begin{equation*}
 |I_{1}|\leq \frac{C_\Omega\| b\|_{L^{\infty}(\Omega)}}{r_0 }\int_{\Omega} |b\omega_{\eps,j}(y)|\,\dd y\leq C_{\Omega,b,r_{0}}|\gamma_{\eps,j}|.
\end{equation*}
The second contribution is controlled in the same way:
\begin{align*}
 |I_2| \leq C_{\Omega,b,r_{0}}|\gamma_{\eps,j}|,
\end{align*}
upon using $\inf_{\Omega_{r_0}} b>0$ with $\Omega_{r_0}$ as defined in \eqref{eq:Omega delta}.
The third contribution is bounded by
\begin{equation*}
 |I_3|\leq C \|\nabla_x R_{\Omega,b}\|_{L^{\infty}(\cal C_{\rho_{b},i}\times \cal C_{\rho_{b},j})}\gamma_{\eps,j}\leq C|\gamma_{\eps,j}|,
\end{equation*}
where we used \eqref{est:R}. These three estimates give the desired bound for $\nabla \Psi_{\varepsilon,j}$.

Next, we prove the Lipschitz property. By the same arguments and the $C^2$-regularity of $b$, the Lipschitz regularity is obvious for $I_{1}$ and $I_{2}$ because for all $y\in\cal C_{\rho_{b},j}$, the function $x\mapsto \ln |x-y|$ belongs to $C^2(\cal C_{\rho_{b},i})$, with its $W^{2,\infty}$ norm bounded by a constant depending only on $r_{0}$, hence independent of $y\in\cal C_{\rho_{b},j}$. This comes from the estimates by below of $|x-y|$. To finish this proof, we only need to prove the Lipschitz regularity for $I_{3}$ which will be the consequence of the following 
\begin{equation*}
 \|\nabla_x^2R_{\Omega,b}\|_{L^{\infty}(\cal C_{\rho_{b},i}\times \cal C_{\rho_{b},j})}\leq C_{\Omega,b,r_{0}}.
\end{equation*}
As the term depending on $G_{\Omega}(\cdot,y)-\frac1{2\pi}\ln |\cdot-y|$ in the definition of $R_{\Omega,b}$ is harmonic, see the argument following \eqref{est:R}, it remains only to prove the previous inequality for $S_{\Omega,b}$. We fix $y\in \cal C_{\rho_{b},j}$ and we recall that $x\mapsto S_{\Omega,b}(x,y)$ satisfies the elliptic problem \eqref{eq:Sb}, where the right hand side term is not singular when $(x,y)\in \cal C_{\rho_{b},i}\times \cal C_{\rho_{b},j}$. Hence, we adapt the argument used to derive \eqref{est:R}: as $b(x)\geq C>0$ on $\Omega_{r_{0}/2}$, the elliptic problem $\diver(b^{-1} \nabla\cdot )$ is non singular, and standard elliptic estimates give 
\begin{align*}
 \| S_{\Omega,b}(\cdot,y) \|_{W^{3,4}(\cal C_{\rho_{b},i})} &\leq C_{r_{0},p} \Big( \|G_{\Omega}(\cdot,y)\|_{W^{1,4}(\cal C_{\rho_{b},i}+B(0,\frac{r_{0}}2))} +\| S_{\Omega,b}(\cdot,y) \|_{L^{2}(\Omega)} \Big)\\
 & \leq C_{r_{0}},
\end{align*}
independently of $y\in \cal C_{\rho_{b},j}$. By the Sobolev embedding $W^{1,4}\hookrightarrow L^\infty$, this ends the proof.
\end{proof}

We are now in position to prove Theorem~\ref{thm:main}.
\begin{proof}[Proof of Theorem~\ref{thm:main}]
Let $T>0$ be fixed. In view of Assumption~\ref{ass: initial data} and the regularity of $(\veps,\weps)$, we deduce that $T_{\eps}>0$, by the definition of $\cal C_{\rho_{b},i}$ in Remark~\ref{rem:r0} and the definition of $T_{\eps}$ in \eqref{eq:Teps main}. As observed just above, we infer from Lemma~\ref{lem: bound F} that $F_{\eps,i}$ defined in \eqref{eq:Feps} satisfies Assumption~\ref{ass: F} for all $t\in [0,T_{\eps})$ and $i\in\{1,\dots,{N_v}\}$.
Therefore, we may apply Theorem~\ref{thm: main red system} to $\omega_{\eps,i}^0$, $z_i^0\in \mathcal{C}_i$, $\gamma_{\eps,i}$ and $F_{\eps,i}$. We conclude first that the $i$-th vortex blob is strongly localized in the direction of steepest ascent, namely for every $k\in (0,1/4)$, there exist $\eps_{k,T}= \min_{i}\eps_{k,T,i}>0$ and $C_{k,T}=\max_{i}C_{k,T,i}>0$ depending only on $(z_i^0)_{i=1,\dots, {N_v}}$, $(\gamma_i)_{i=1,\dots, {N_v}}$, $(\Gamma_i)_{i=1,\dots, {N_{is}}}$, $M_{0}$, $b$, $\Omega$, $k$ and $T$ such that, for any $\varepsilon\in (0,\varepsilon_{k,T}]$, 
\begin{equation*}
 \supp \omega_{\varepsilon,i}(t,\cdot) \subset \left\{ x \, : \, |b(x)-b(z_i^0)| \leq \frac{C_{k,T}}{|\ln \varepsilon|^k}\right\}
 \end{equation*}
for all $t\in[0,T_{\eps}]$ and $i\in \{1,\dots,{N_v}\}$. Up to choosing, $\varepsilon_{\frac18,T}$ smaller if necessary, we can assume that $\frac{C_{\frac18,T}}{|\ln \varepsilon_{\frac18,T}|^{\frac18}}\leq \rho_{b}/2$.
 As $\omega_{\varepsilon,i}(t,\cdot)$ is constant along continuous curves from $B(z^0,\varepsilon)\subset \cal C_{\rho_{b}/2}$, see Remark~\ref{rem:r0}, we infer that $\supp \omega_{\varepsilon,i}(t,\cdot)\subset \cal C_{\rho_{b}/2}$. By definition of $T_{\varepsilon}$ \eqref{eq:Teps main}, we obtain that $T_{\varepsilon}=T$ for every $\varepsilon\in (0, \eps_{\frac18,T}]$.

The point {\it (i)} of Theorem~\ref{thm:main} is proved by the weak localization and the limit of trajectories, when we consider $T_{\varepsilon}=T$ for every $\varepsilon\in (0, \eps_{\frac18,T}]$.

In order to show item {\it (ii)} , namely the strong localization on $[0,T]$ for all $k\in (0,1/4)$, it suffices to replace $\eps_{k,T}$ by $\min(\eps_{k,T},\eps_{\frac18,T})$, which completes the proof of Theorem~\ref{thm:main}.
\end{proof}

The remaining part of this paper is dedicated to the proof of Theorem~\ref{thm: main red system}. 
 More precisely, Section~\ref{sec:BS} details the Biot-Savart type decomposition of $\veps$ and provides first estimates, while the energy estimates are proven in Section~\ref{sec:energy}. Section~\ref{sec:momemtum} is devoted to bound the moment of inertia and to establish the weak localization property of Theorem~\ref{thm: main red system} for short time. In Section~\ref{sec:transverse} we prove the strong localization property in the transverse direction, namely {\it(i)} of Theorem~\ref{thm: main red system}, which will imply that the weak localization holds on $[0,T]$. The asymptotic trajectory is derived in Section~\ref{sec:trajectory}.

\begin{remark}
Let us notice that, up to changing the sign of $\omega_{\varepsilon}$, $v_{\varepsilon}$ and $t$, it is enough to prove Theorem~\ref{thm: main red system} for non-negative vorticities, namely $\delta=1$ in Assumption~\ref{data red}. Hence, we will always use in the sequel $|\omega_{\varepsilon}|=\omega_{\varepsilon}$ for every $t,x$ and $\varepsilon$.
\end{remark}

\section{Decomposition of the Biot-Savart law}\label{sec:BS}
The aim of this section is to introduce a suitable Hodge decomposition for the velocity field which the subsequent sections build upon. 

Throughout this and the subsequent sections, we work under the hypothesis of Theorem~\ref{thm: main red system}. For the convenience of the reader, we summarize the setting of the solutions to \eqref{eq:lake} under consideration.

\begin{assumption}\label{ass: main part}
The lake $(\Omega,b)$ satisfies Assumption~\ref{assum:lake}, the initial data $\weps^0$ are sharply concentrated in the sense of Assumption~\ref{data red} (for $\delta=1$) and $F_{\eps}$ satisfies Assumption~\ref{ass: F}. Let $T>0$, we denote by $(\omega_\eps, v_\eps)$ the unique corresponding weak solution of \eqref{eq:transport-1} in the sense of Proposition~\ref{prop:well-posed} on $[0,T]$, where $v$ in (ii) has to be replaced by $\frac1{\lne}(v_{\varepsilon}+F_{\varepsilon})$. This solution verifies Proposition~\ref{prop:transport} with $\omega=\omega_{\varepsilon}$ and $v=\frac1{\lne}(v_{\varepsilon}+F_{\varepsilon})$, but where we note that $K_{T}$ in the first item depends on $\varepsilon$.

We consider the neighborhood $\cal C_{\rho_b}$ of $\cal C$ as defined in Remark~\ref{rem:r0}.
Analogously to \eqref{eq:Teps main}, we define the time
\begin{equation}\label{def:maxi-T}
 T_{\eps}=\sup\Big\{ t\in[0,T] \, \, : \, \supp \omega_{\eps}(s,\cdot) \subset \cal C_{\rho_{b}} \ \text{for all } \, s\in[0,t] \Big\}.
\end{equation}
\end{assumption}

Note that in view of Remark~\ref{rem:r0} one has $\dist(\cal C_{\rho_{b}},\partial\Omega)> r_0$ hence $ \cal C_{\rho_{b}} \subset \Omega_{r_{0}}$, where $\Omega_{r_{0}}$ is defined in \eqref{eq:Omega delta}.

First, we introduce a suitable decomposition of the velocity field $\veps$ in terms of stream functions. We recall that $v_{\varepsilon}$ is the vector field given by the Biot-Savart type law \eqref{BS:is} and \eqref{eq:BS}. The expansion of the Green kernel \eqref{eq.Gdecomp} allows one to decompose the velocity field $\ueps$ as
\begin{equation}\label{eq:decomposition}
 \ueps=v_{\varepsilon,K}+v_{\varepsilon,L}+v_{\varepsilon,R},
\end{equation}
where we define
\begin{itemize}
 \item the most singular term as the 2D Biot-Savart law:
 \begin{equation}\label{eq:uK}
\begin{aligned}
 v_{\varepsilon,K}(x)&=\frac{1}{2\pi b(x)}\int_{\Omega}\nabla_x^{\perp}(\ln|x-y|)\sqrt{b(x)b(y)}(b\weps)(y) \,\dd y\\
 &=\frac{1}{2\pi b(x)} \int_{\Omega}K(x,y)\sqrt{b(x)b(y)}(b\weps)(y) \,\dd y,
 \end{aligned}
\end{equation}
where we have denoted
\[
K(x,y):=\frac{(x-y)^{\perp}}{|x-y|^2}.
\]
Note that while being singular $v_{\varepsilon,K}$ has a symmetric structure and will give the standard spinning around the point vortex which will not contribute to the displacement of the vortex core;
\item the intermediate vector field which will account for the main dynamics:
\begin{equation}\label{eq:uL}
\begin{aligned}
 & v_{\varepsilon,L} (x) =\frac{\nabla^\perp b(x)}{4\pi b^2(x)} \psi_{\varepsilon}(x) \\
 &\psi_{\varepsilon}(x) =\int_{\Omega}\ln|x-y| \sqrt{b(x)b(y)} (b\weps)(y)\,\dd y;
 \end{aligned}
\end{equation}
\item and the remainder term
\begin{equation}\label{eq:uR}
\begin{aligned}
 v_{\varepsilon,R}(x)=&\frac{1}{b(x)}\int_{\Omega} \nabla_x^\perp R_{\Omega,b}(x,y) (b\omega_{\varepsilon})(y)\,\dd y \\
 &+ \frac{1}{b(x)} \sum_{k=1}^{N_{is}} \Big(\Gamma_{\varepsilon}^k + \int_{\Omega}b\omega_{\varepsilon}\varphi^k \Big)\nabla^\perp \psi^k.
 \end{aligned}
\end{equation}
\end{itemize}

Note that in view of \eqref{est:phipsi} and \eqref{est:R}, it follows that for all $t\in[0,T_{\eps})$ it holds 
\begin{equation}\label{bound:vR}
\|v_{\varepsilon,R}\|_{L^\infty(\cal C_{\rho_{b}})}\leq C_{r_{0}} \int_{\Omega} (b\omega_\eps)(t,y)\,\dd y +C_{r_{0}}|\Gamma_{\varepsilon}|\leq C_{r_{0}} (\gamma + |\Gamma|) .
\end{equation}

A first consequence of the decomposition \eqref{eq:decomposition} is the following bound.

\begin{lemma}\label{lemma:J} Under Assumption~\ref{ass: main part}, for $k\geq 1$, let 
$$\mathcal{J}_k(t)=\int_{\Omega}b(x)^k (b\weps)(t,x)\,\dd x.$$
Then there exists $C_{k}>0$ independent of $\varepsilon$ such that 
\begin{equation*}
|\mathcal{J}_k(t)-\gamma_{\varepsilon} b(z^0)^k|\leq\frac{C_k}{\lne} \quad \forall t\in [0,T_{\varepsilon}).
\end{equation*}
\end{lemma}
\begin{proof}
We use the weak formulation of \eqref{eq:transport-1} (see Proposition~\ref{prop:well-posed}) with test function $b(x)^k \chi(x)$ where $\chi\in C^\infty_{c}(\Omega)$ is a smooth cutoff function such that $\chi\equiv 1$ on $\cal C_{\rho_{b}}$, so that $b^k \chi$ is $C^1$ by definition of $\cal C_{\rho_{b}}$ and $T_{\varepsilon}$.
We have for a.e. $t\in [0,T_\eps]$:
\begin{align*}
\frac{\dd}{\dd t}\mathcal{J}_k(t) &=\frac{\dd}{\dd t}\int_{\Omega}b(x)^k \chi(x) (b\weps)(t,x)\,\dd x \\
&=\frac1\lne \int_{\Omega} \nabla\Big(b(x)^k \chi(x) \Big)\cdot (\veps+F_{\varepsilon})(t,x)(b\weps)(t,x)\,\dd x\\
&=\frac1\lne \int_{\Omega} \nabla\Big(b(x)^k \Big)\cdot (\veps+F_{\varepsilon})(t,x)(b\weps)(t,x)\,\dd x.
 \end{align*}
 This use of \eqref{eq:transport-1} will be systematic in the sequel, and we will not mention anymore the use of $\chi$ and that the equalities hold true almost everywhere.

The decomposition \eqref{eq:decomposition} together with the estimates of $v_{\varepsilon,R}$ \eqref{bound:vR} and of $F_\eps$, see Assumption~\ref{ass: F}, allow us to compute
\begin{align*}
\frac{\dd}{\dd t}\mathcal{J}_k(t)
=& \int_{\Omega} \frac{v_{\eps,K}(t,x)\cdot\nabla b^k(x)}{\lne} (b\weps)(t,x)\,\dd x +\mathcal{O}\left(\frac{1}{\lne}\right)\\
=&\frac1{2\pi \lne}\iint_{\Omega^2} K(x,y) \cdot \Big( \frac{\nabla b^k(x)}{b(x)}\Big)\sqrt{b(x)b(y)}(b\weps)(x)(b\weps)(y)\,\dd x\dd y\\
&+ \mathcal{O}\left(\frac{1}{\lne}\right)\\
=&\frac1{4\pi \lne}\iint_{\Omega^2} K(x,y) \cdot \Big( \frac{\nabla b^k(x)}{b(x)}-\frac{\nabla b^k(y)}{b(y)}\Big)\\
&\quad\times\sqrt{b(x)b(y)}(b\weps)(x)(b\weps)(y)\,\dd x\dd y + \mathcal{O}\left(\frac{1}{\lne}\right)
\end{align*}
by symmetrizing with respect to $x$ and $y$. From the application of the mean-value theorem \eqref{MVT} to $\frac{\nabla b^k}b$ on $\Omega_{r_{0}}$ we get
\begin{equation*}
\Big|\frac{\nabla b^k(x)}{b(x)}-\frac{\nabla b^k(y)}{b(y)}\Big|\leq \widetilde{C}_{r_{0}} |x-y|.
\end{equation*}
We deduce that
\[
\left|\frac{\dd}{\dd t}\mathcal{J}_k(t)\right|\leq \frac{C}{\lne},
\]
which implies that
\[
|\mathcal{J}_k(t)-\gamma_{\varepsilon} b(z^0)^k|\leq |\mathcal{J}_k(0)-\gamma_{\varepsilon} b(z^0)^k| + \frac{CT}{\lne} \leq C\eps+ \frac{C}{\lne},
\]
by virtue of Assumption~\ref{data red} on the initial data $\weps^0$.
\end{proof}

\begin{corollary}\label{coro:mai}
Under Assumption~\ref{ass: main part} we have
\begin{equation*}
\int_{\Omega} \omega_\eps(t,x)\,\dd x = \frac{\gamma_\eps}{b(z^0)}+\mathcal{O}\left(\frac{1}{\lne}\right),\quad \quad \forall t\in [0,T_{\varepsilon}).
\end{equation*}
\end{corollary}
\begin{proof}
We expand the mass as
\begin{align*}
\int_{\Omega} \omega_\eps(t,x)\,\dd x =& \int_{\Omega} (b\omega_\eps)(t,x)\left(\frac{1}{b(x)}- \frac{1}{b(z^0)}\right)\,\dd x
+\frac{\gamma_\eps}{b(z^0)}\\
=&\int_{\Omega} (b\omega_\eps)(t,x)\frac{(b(z^0)-b(x))^2}{b(x)b(z^0)^2}\, \dd x\\
&+ \frac{1}{b(z^0)^2}\int_{\Omega} (b\omega_\eps)(t,x)\left(b(z^0)-b(x)\right)\,\dd x+\frac{\gamma_\eps}{b(z^0)}.
\end{align*}
We observe that by the definition of $T_{\varepsilon}$ \eqref{def:maxi-T} we have $b(x)\geq \min_{\Omega_{r_0}} b>0$ for $x\in \text{supp}(\omega_\eps)(t,\cdot)$ therefore
\begin{align*}
\int_{\Omega} (b\omega_\eps)(t,x)\frac{(b(z^0)-b(x))^2}{b(x)b(z^0)^2}\, \dd x
&\leq C \int_{\Omega} (b\omega_\eps)(t,x)(b(z^0)-b(x))^2\, \dd x\\
&\leq C\left(b(z^0)^2 \gamma_\eps+ \mathcal{J}_2(t)-2b(z^0)\mathcal{J}_1(t)\right)\\
&\leq \frac{C}{\lne},
\end{align*}
where we have applied Lemma~\ref{lemma:J}.

Similarly,
\begin{align*}
\left|\frac{1}{b(z^0)^2}\int_{\Omega} (b\omega_\eps)(t,x)\left(b(z^0)-b(x)\right)\,\dd x\right|
&=\frac{1}{b(z^0)^2}\left| b(z^0)\gamma_\eps-\mathcal{J}_1(t)\right|\leq \frac{C}{\lne}.
\end{align*}
This concludes our proof.
\end{proof}

\section{Estimates on the energy and the stream function}\label{sec:energy}
Working again in the setting of Assumption~\ref{ass: main part}, we aim to introduce suitable bounds on the local energy defined by
\begin{align*}
E_{\varepsilon}(t)=&\int_{\Omega} b(x)| b^{-1}(x)\nabla^\perp \Psi_{\varepsilon}(t,x)|^2\,\dd x \\
=&\int_{\Omega} \nabla^\perp \Psi_{\varepsilon}(t,x) \cdot b^{-1}(x)\nabla^\perp \Psi_{\varepsilon}(t,x)\,\dd x 
= -\int_{\Omega} \Psi_{\varepsilon}(t,x) (b\omega_{\varepsilon})(t,x)\,\dd x \\
=& -\iint_{\Omega^2} G_{\Omega,b}(x,y) (b\omega_{\varepsilon})(t,x)(b\omega_{\varepsilon})(t,y)\,\dd x\dd y \\
\end{align*}
and on the stream function $\psi_{\varepsilon}$ defined in \eqref{eq:uL}.
In the absence of external field and islands, the local energy coincides with the total energy and it is therefore conserved, see the proof of Proposition~\ref{prop:energy}.

First, we relate the local energy and the stream function $\psi_{\eps}$ in the following lemma.
\begin{lemma}\label{lemma:energy-1}
Under Assumption~\ref{ass: main part}, we have for all $t\in [0,T_{\varepsilon})$
\begin{align*}
E_{\varepsilon}(t)&=-\frac1{2\pi}\int_{\Omega}\psi_\eps(t,x)(b\weps)(t,x)\,\dd x+\mathcal{O}(1).
\end{align*}
\end{lemma}

\begin{proof}
By definition of $R_{\Omega,b}$, see \eqref{eq.Gdecomp}, it is clear that
\begin{multline*}
E_{\varepsilon}(t)+\frac1{2\pi}\int_{\Omega}\psi_\eps(t,x)(b\weps)(t,x)\,\dd x \\= -\iint_{\Omega^2} R_{\Omega,b}(x,y) (b\omega_{\varepsilon})(t,x)(b\omega_{\varepsilon})(t,y)\,\dd x\dd y
\end{multline*}
the conclusion follows from \eqref{est:R}.
\end{proof}

We use this relation to establish a lower bound for the initial energy.
\begin{lemma}\label{lemma:energy-2}
Under Assumption~\ref{ass: main part}, it holds
\[
E_{\varepsilon}(0)\geq \frac{1}{2\pi} \gamma_{\varepsilon}^2 b(z^0)\lne+\mathcal{O}(1).
\]
\end{lemma}

\begin{proof}
By Lemma~\ref{lemma:energy-1}, it suffices to bound
\[
-\frac1{2\pi}\int_{\Omega}\psi_\eps^0(x)(b\weps^0)(x)\,\dd x.
\]
Exploiting that the initial data $\weps^0$ is sharply concentrated, see Assumption~\ref{data red}, and the mean value theorem \eqref{MVT}, we have
\begin{equation*}
 \left|\sqrt{b(x)}-\sqrt{b(y)}\right|\leq C_{\Omega,r_0}\|\nabla b\|_{L^{\infty}(\Omega_{r_{0}/2})}|x-y|\leq C_{\Omega,b,r_0} |x-y|,
\end{equation*}
for all $x,y\in \supp(\weps^0)$, where we recall that $\supp(\weps^0)\subset\cal C_{\rho_b}\subset \Omega_{r_{0}}$ for all $\eps\in (0,\varepsilon_{0}]$, see Remark~\ref{rem:r0}.
It follows that
\begin{align*}
\int_{\Omega}\psi_\eps^0(x)(b\weps^0)(x)\,\dd x=&\iint_{\Omega^2}\ln |x-y| \sqrt{b(x)b(y)}(b\omega_\eps^0)(x)(b\omega_\eps^0)(y)\,\dd x\dd y\\
=&\iint_{\Omega^2}\ln |x-y| b(x)(b\omega_\eps^0)(x)(b\omega_\eps^0)(y)\,\dd x\dd y+\mathcal{O}(1)\\
=&\iint_{\Omega^2}\ln |x-y| (b(x)-b(z^0) )(b\omega_\eps^0)(x)(b\omega_\eps^0)(y)\,\dd x\dd y\\
&+b(z^0)\iint_{\Omega^2}\ln |x-y|(b\omega_\eps^0)(x)(b\omega_\eps^0)(y)\,\dd x\dd y+\mathcal{O}(1) .
\end{align*}
On the one hand, we use the localization assumption of $\omega^0_{\varepsilon}$ to get
\begin{align*}
\Big|\iint_{\Omega^2}&\ln |x-y| (b(x)-b(z^0) )(b\omega_\eps^0)(x)(b\omega_\eps^0)(y)\,\dd x\dd y\Big|\\
&\leq C\|\nabla b\|_{L^\infty(\Omega_{r_{0}/2})} M_{0}\varepsilon \iint_{\Omega^2} (-\ln |x-y|) (b\omega_\eps^0)(x)(b\omega_\eps^0)(y)\,\dd x\dd y \\
&\leq C\varepsilon \int_{\Omega}(b\omega_\eps^0)(y) \Big( \int_{B(0,2M_{0}\varepsilon)} (-\ln |z|) \|b\|_{L^\infty} \frac{M_{0}}{\varepsilon^2} \,\dd z\Big)\,\dd y \\
&\leq C \eps\lne.
\end{align*}

On the other hand, we use that $|x-y|\leq 2M_{0}\varepsilon$ for all $x,y\in \supp \omega_{\varepsilon}^0$ to estimate
\[
-\frac{b(z^{0})}{2\pi}\iint_{\Omega^2} \ln |x-y| (b\omega_\eps^0)(x)(b\omega_\eps^0)(y)\,\dd x\dd y \geq -\frac{b(z^{0})}{2\pi}\ln (2M_{0}\varepsilon) \gamma_{\varepsilon}^2,
\]
which completes the proof.
\end{proof}

\begin{remark}
Since the proof of Lemma~\ref{lemma:energy-2} relies on the sharp concentration of the initial data $\weps^0$, namely that $\supp(\weps^0)\subset B(z^0,M_0\eps)$, we do not claim that it extends to positive times. Nevertheless, by the slow variation property of the local energy, we will deduce that this lower bound holds for any time, see Proposition~\ref{prop:energy} below. 
\end{remark}

The purpose of the next lemma is to provide an upper bound for $\psi_\eps$ as defined in \eqref{eq:uL} which in turn will yield a precise estimate of the initial energy by virtue of Lemma~\ref{lemma:energy-1} and Lemma~\ref{lemma:energy-2}. The lemma is stated for arbitrary times as such an estimate is required below. 

\begin{lemma}\label{lem:est-psi} 
Under Assumption~\ref{ass: main part} and for $\eps_{0}$ sufficiently small, there exists $C>0$ such that
\[
-C\leq -\psi_\eps(t,x)\leq \gamma_{\varepsilon} \frac{b(x)^2}{b(z^0)}\lne +C
\]
for any $x\in \cal C_{\rho_b}$ and $t\in[0,T_\eps)$.
\end{lemma}
\begin{proof}
We infer from the mean value theorem \eqref{MVT} that
\begin{align*}
 -\int_{\Omega}\ln|x-y|& \sqrt{b(x)b(y)} (b\omega_\eps)(t,y)\,\dd y\\
=& -b(x)^2\int_{\Omega}\ln|x-y| \omega_\eps(t,y)\,\dd y\\
& - \sqrt{b(x)}\int_{\Omega} \ln|x-y| (b^{3/2}(y) -b^{3/2}(x) )\omega_\eps(t,y)\,\dd y\\
=&-b(x)^2\int_{\Omega}\ln\left(\frac{|x-y|}{{\rm diam\,} \Omega}\right) \omega_\eps(t,y)\,\dd y+\mathcal{O}(1),
\end{align*}
where ${\rm diam\,} \Omega=\max_{\Omega^2}|x-y|$. This equality is sufficient to get $-C\leq - \psi_\eps(t,x)$. We use now the rearrangement of the mass, namely we apply Lemma~\ref{lem:rearrang0} for $g(s):=-\ln (s/{\rm diam\,} \Omega)\mathds{1}_{s\leq {\rm diam\,}\Omega}$ and $\gamma:=\widetilde{\gamma_\eps}=\int_{\Omega}\omega_\eps(t,x)\, \dd x$, to obtain
\[
- \psi_\eps(t,x)\leq b(x)^2 \frac{2\pi M_{0}}{\varepsilon^2} \int_{0}^{R_{0}} sg(s)\,\dd s +\mathcal{O}(1)
\]
with
\[
 \pi \frac{M_0}{\eps^2} R_{0}^2=\widetilde{\gamma_{\varepsilon}}=\int_{\Omega}\omega_\eps(t,x)\, \dd x
\] 
(in particular $R_0<1$ for $\varepsilon_{0}$ small enough).
Hence
\begin{align*}
- \psi_\eps(t,x)&\leq \frac{2\pi M_{0}}{\varepsilon^2} b(x)^2 \Big(-\frac12 R_{0}^2 \ln R_{0} + \frac14 R_{0}^2 + \frac{R_{0}^2}{2}\ln \rm diam\, \Omega\Big)+\mathcal{O}(1) \\
&\leq \widetilde{\gamma_{\varepsilon}} b(x)^2 |\ln \varepsilon| + \mathcal{O}(1).
\end{align*}
In particular, we have just proved that
\begin{equation}
\label{est:toulouse-19}
-\int_{\Omega}\ln\left(\frac{|x-y|}{{\rm diam\,} \Omega}\right) (b\omega_\eps)(t,y)\,\dd y
\leq C\lne.
\end{equation}
Applying now Corollary~\ref{coro:mai}, we conclude that
\begin{align*}
- \psi_\eps(t,x)
&\leq \gamma_{\varepsilon} \frac{b(x)^2}{b(z_0)} |\ln \varepsilon| + \mathcal{O}(1).
\end{align*}
\end{proof}

Combining Lemmas~\ref{lemma:energy-1} and \ref{lemma:energy-2} with Lemma~\ref{lem:est-psi} enables us to determine the leading order term of the initial energy.

\begin{corollary}\label{coro:energy0}
Under Assumption~\ref{ass: main part} and for $\varepsilon_{0}$ small enough, we have
\begin{equation*}
E_{\varepsilon}(0)=\frac{1}{2\pi} \gamma_{\varepsilon}^2 b(z^0)\lne+\mathcal{O}(1)
\end{equation*}
for all $\eps\in (0,\eps_0)$.
\end{corollary}
\begin{proof}
Lemma~\ref{lemma:energy-2} yields the lower bound. To compute the upper bound, we invoke Lemmas~\ref{lemma:energy-1} and \ref{lem:est-psi} to obtain
\begin{align*}
 E_{\varepsilon}(0)&=-\frac1{2\pi}\int_{\Omega}\psi_{\eps}^0(x)(b\weps^0)(x)\,\dd x\\
 &\leq \frac{\gamma_{\varepsilon}|\ln\eps|}{2\pi b(z^0)}\int_{\Omega}b(x)^2(b\weps^0)(x)\,\dd x+C\gamma_{\varepsilon}\\
 &\leq\frac{\gamma_{\varepsilon}^2}{2\pi}b(z^0)|\ln\eps|+\frac{\gamma_{\varepsilon}}{2\pi b(z^0)}|\ln\eps|
 \int_{\Omega}(b(x)^2-b(z^0)^2)(b\weps^0)(x) \,\dd +C\gamma_{\varepsilon}\\
 &\leq\frac{\gamma_{\varepsilon}^2}{2\pi}b(z^0)|\ln\eps|+C\gamma_{\varepsilon}^2\eps|\ln\eps|+\mathcal{O}(1),
\end{align*}
where we have used the initial localization of the vorticity, see Assumption~\ref{data red} and \eqref{MVT}.
\end{proof}

Next, we prove the conservation of the local energy at leading order up to time $T_{\eps}$.

\begin{proposition}\label{prop:energy}
Under Assumption~\ref{ass: main part} and for $\varepsilon_{0}$ small enough, there exists $C>0$ such that for all $t\in [0,T_\eps)$ we have
\[
\Big| E_{\varepsilon}(t) - \frac{1}{2\pi} \gamma_{\varepsilon}^2 b(z^0)\lne \Big| \leq C.
\]
\end{proposition}
\begin{proof}
As $\sqrt b v_{\varepsilon}\in C(\R_{+}; L^2(\Omega))$, we can give a sense in $\cal D'(0,T_{\varepsilon})$ to the time-derivative of $E_{\eps}$ by arguing similarly as detailed in the proof of Lemma~\ref{lemma:J}. We compute
\begin{align*}
 \frac{\dd}{\dd t}E_{\varepsilon}(t) 
 =&-\frac2{\lne}\iint_{\Omega^2} \nabla_{x} G_{\Omega,b}(x,y)\cdot (v_{\varepsilon}+F_{\varepsilon})(t,x) (b\omega_{\varepsilon})(t,x)(b\omega_{\varepsilon})(t,y)\,\dd x\dd y\\ 
 =&-\frac2{\lne} \sum_{k=1}^{N_{is}} \Big(\Gamma_{\varepsilon}^k + \int_{\Omega}b\omega_{\varepsilon}\varphi^k \Big)\\
 &\qquad\qquad \times\iint_{\Omega^2} \nabla_{x} G_{\Omega,b}(x,y)\cdot\nabla^\perp \psi^k(x)(b\omega_{\varepsilon})(t,x)(b\omega_{\varepsilon})(t,y)\,\dd x\dd y \\
&-\frac2{\lne}\iint_{\Omega^2} \nabla_{x} G_{\Omega,b}(x,y)\cdot F_{\varepsilon}(t,x) (b\omega_{\varepsilon})(t,x)(b\omega_{\varepsilon})(t,y)\,\dd x\dd y 
\end{align*}
because 
\[
b(x)v_{\varepsilon}(t,x) =\Big(\int_{\Omega} \nabla_{x} G_{\Omega,b}(x,z) (b\omega_{\varepsilon})(z)\,\dd z \Big)^\perp+ \sum_{k=1}^{N_{is}} \Big(\Gamma_{\varepsilon}^k + \int_{\Omega}b\omega_{\varepsilon}\varphi^k \Big)\nabla^\perp \psi^k.
\]

By the decomposition of the Green kernel \eqref{eq.Gdecomp}, we deduce easily from \eqref{est:phipsi}, \eqref{est:R} and Assumption~\ref{ass: F} that
\begin{gather*}
\Big|\frac2{\lne}\iint_{\Omega^2} \nabla_{x} R_{\Omega,b}(x,y)\cdot \nabla^\perp \psi^k(x) (b\omega_{\varepsilon})(t,x)(b\omega_{\varepsilon})(t,y)\,\dd x\dd y \Big| \leq \frac{C}{\lne}, \\
 \Big|\frac2{\lne}\iint_{\Omega^2} \nabla_{x} R_{\Omega,b}(x,y)\cdot F_{\varepsilon}(t,x) (b\omega_{\varepsilon})(t,x)(b\omega_{\varepsilon})(t,y)\,\dd x\dd y \Big| \leq \frac{C}{\lne}. 
 \end{gather*}

Next, we compute by the skew-symmetry of $K$ combined with the mean value theorem \eqref{MVT} and Assumption~\ref{ass: F}
\begin{align*}
 \Big|\frac1{\pi\lne}&\iint_{\Omega^2} K(x,y)^\perp \cdot F_{\varepsilon}(t,x) \sqrt{b(x)b(y)} (b\omega_{\varepsilon})(t,x)(b\omega_{\varepsilon})(t,y)\,\dd x\dd y \Big| \\
 =& \Big|\frac1{2\pi\lne}\iint_{\Omega^2} K(x,y)^\perp \cdot \Big( F_{\varepsilon}(t,x)- F_{\varepsilon}(t,y)\Big)\\
 &\hspace{5cm} \times \sqrt{b(x)b(y)} (b\omega_{\varepsilon})(t,x)(b\omega_{\varepsilon})(t,y)\,\dd x\dd y \Big| \\
 \leq& \frac{C}{\lne}.
 \end{align*}
The last term is treated using Assumption~\ref{ass: F}
 \begin{align*}
 \Big|\frac1{\pi\lne}&\iint_{\Omega^2} \ln |x-y| \frac{\nabla b(x)}{2b(x)} \cdot F_{\varepsilon}(t,x) \sqrt{b(x)b(y)} (b\omega_{\varepsilon})(t,x)(b\omega_{\varepsilon})(t,y)\,\dd x\dd y \Big|\\
&\leq \frac{C}{\lne}\int_{\Omega} (b\omega_{\varepsilon})(t,y) \Big( \int_{\Omega} | \ln |x-y| | (b\omega_{\varepsilon})(t,x)\,\dd x \Big)\,\dd y \\
&\leq \frac{C}{\lne}\int_{\Omega} (b\omega_{\varepsilon})(t,y) \Big( \int_{\Omega} - \ln \frac{|x-y|}{{\rm diam\,}\Omega} (b\omega_{\varepsilon})(t,x)\,\dd x \Big)\,\dd y +\cal O\Big(\frac1\lne\Big)\\
&\leq \cal O(1)
 \end{align*}
 where we have used \eqref{est:toulouse-19}. These two last estimates can be performed in the same way replacing $F_{\varepsilon}$ by $\nabla^\perp\psi^k$.
 
 With these estimates, we conclude that for all $t\in [0,T_{\varepsilon})$
 \[
 \Big| \frac{\dd}{\dd t}E_{\varepsilon}(t) \Big| \leq C
 \]
which means that $E_{\varepsilon}(t)=E_{\varepsilon}(0) +\mathcal{O}(1)$. Corollary~\ref{coro:energy0} allows us to conclude.
\end{proof}

\begin{remark}\label{rem:LipF}
We emphasize that throughout the paper we only rely once on Lipschitz property of $F_{\eps}$ stated in Assumption~\ref{ass: F}, namely to infer the expansion of the energy in Proposition~\ref{prop:energy}. Note that the proof adapts to exterior fields with Lipschitz constant of order $\mathcal{O}(\lne)$ instead of $\mathcal{O}(1)$ as considered in Assumption~\ref{ass: F}. 

The approach presented is hence robust enough to be adapted to different scaling regimes leading to an exterior field which respects such a bound. Note, however, that this is insufficient in order to investigate phenomena such as leapfrogging vortex rings and similar behavior for the lake equations.
\end{remark}

This section is completed by an estimate on the stream function $\psi_\eps$ that is required in the sequel. 

\begin{lemma}\label{lemma:L2}
Under Assumption~\ref{ass: main part} and for $\varepsilon_{0}$ small enough, there exists $C>0$ such that for all $t\in [0,T_\eps)$ we have
\[
\int_{\Omega}\left |\gamma_{\varepsilon}\psi_\eps(t,x)-\int_{\Omega}\psi_\eps(t,y)(b\weps)(t,y)\,\dd y \right|^2 (b\weps)(t,x)\,\dd x
\leq C\lne.
\]
\end{lemma}

\begin{proof}
We have by Cauchy-Schwarz inequality 
\begin{align*}
\frac1{\gamma_{\varepsilon}}&\int_{\Omega}\left (\int_{\Omega}\left(\psi_\eps(x)-\psi_\eps(y)\right)(b\weps)(y)\,\dd y \right)^2 (b\weps)(x)\,\dd x\\
=&\frac1{\gamma_{\varepsilon}}\int_{\Omega}\left (\int_{\Omega}\left(\psi_\eps(x)-\psi_\eps(y)\right)\sqrt{(b\weps)(y)}\sqrt{(b\weps)(y)}\,\dd y \right)^2 (b\weps)(x)\,\dd x\\
\leq& \iint_{\Omega^2}\left (\psi_\eps(x)-\psi_\eps(y)\right)^2(b\weps)(y) (b\weps)(x)\,\dd x\dd y\\
\leq&\iint_{\Omega^2}\psi_\eps(x)^2(b\weps)(y) (b\weps)(x)\,\dd x\dd y+\iint_{\Omega^2}\psi_\eps(y)^2(b\weps)(y) (b\weps)(x)\,\dd x\dd y\\
&-2\iint_{\Omega^2} \psi_\eps(x) \psi_\eps(y)(b\weps)(y) (b\weps)(x)\,\dd x\dd y\\
\leq&2\gamma_{\varepsilon} \int_{\Omega}\psi_\eps(x)^2 (b\weps)(x)\,\dd x-2\left(\int_{\Omega} \psi_\eps(x) (b\omega_\eps)(x)\,\dd x\right)^2.
\end{align*}
Using Lemma~\ref{lem:est-psi}, we estimate the first term in the right hand side:
\[
2\gamma_{\varepsilon} \int_{\Omega}\psi_\eps(x)^2 (b\weps)(x)\,\dd x 
\leq 2 \frac{\gamma_{\varepsilon}^3}{b(z^0)^2}\lne^2\int_{\Omega}b(x)^4 (b\weps)(x)\,\dd x+ C\lne.
\]
Therefore, Lemma~\ref{lemma:energy-1} implies that
\begin{multline*}
\frac1{\gamma_{\varepsilon}}\int_{\Omega}\left (\int_{\Omega}\left(\psi_\eps(x)-\psi_\eps(y)\right)(b\weps)(y)\,\dd y \right)^2 (b\weps)(x)\,\dd x\\
\leq 2\left( \frac{\gamma_{\varepsilon}^3}{b(z^0)^2}\lne^2\int_{\Omega}b(x)^4 (b\weps)(x)\,\dd x -(2\pi)^2E_{\varepsilon}(t)^2\right)+ C(\lne+E_{\varepsilon}(t)).
\end{multline*}
Proposition~\ref{prop:energy} gives
$$E_{\varepsilon}(t)^2 = \frac{1}{(2\pi)^2} \gamma_{\varepsilon}^4 b(z^0)^2\lne^2+\mathcal{O}(\lne),$$
thus we obtain
\begin{multline*}
\int_{\Omega}\left (\int_{\Omega}\left(\psi_\eps(x)-\psi_\eps(y)\right)(b\weps)(y)\,\dd y \right)^2 (b\weps)(x)\,\dd x\\
\leq 2 \frac{\gamma_{\varepsilon}^4\lne^2}{b(z^0)^2} \left(\int_{\Omega}b(x)^4 (b\weps)(x)\,\dd x-\gamma_{\varepsilon} b(z^0)^4\right)+C\lne.
\end{multline*}
By virtue of Lemma~\ref{lemma:J} for $k=4$ we conclude that
\[
\int_{\Omega}\left (\int_{\Omega}\left(\psi_\eps(x)-\psi_\eps(y)\right)(b\weps)(y)\,\dd y \right)^2 (b\weps)(x)\,\dd x\leq C \lne.
\]
\end{proof}

\section{Estimate on the momentum and weak concentration}\label{sec:momemtum}

Throughout this section, we consider again solutions $(\omega_\eps, v_\eps)$ to \eqref{eq:transport-1} on $[0,T)\times \Omega$ that satisfy Assumption~\ref{ass: main part}. We introduce the center of mass
\begin{equation}\label{def:point-vortex}
z_\eps(t)=\frac{1}{\gamma_{\varepsilon}}\int_{\Omega} x b(x)\weps(t,x)\,\dd x
\end{equation}
and the moment of intertia centered at $z_\eps(t)$ given by
\begin{equation}\label{def:momentum}
\Ie(t)=\int_{\Omega} |x-z_\eps(t)|^2 b(x)\omega_\eps(t,x)\,\dd x,
\end{equation}
both of which are crucial for the study of the evolution of initially sharply concentrated vorticity. Note that in view of Assumption~\ref{ass: main part}, we initially have
\[
\Ie(0) \leq 4\gamma_{\varepsilon}M_{0}^2\varepsilon^2
\]
because $ \supp \omega_{\varepsilon}^0\subset B(z^0,\eps M_0)$ and 
\[
|z^{0}-z_\eps(0)| \leq\frac1{\gamma_{\varepsilon}} \int_{\Omega} |z^{0}-x| (b\omega_{\varepsilon}^0)(x) \,\dd x \leq M_{0} \varepsilon .
\]

As $z_\eps(t)$ belongs to the convex hull of $\cal C_{\rho_b}$ for any $t\in [0,T_{\varepsilon})$, it is not obvious, without adding an extra geometrical assumption on the lake, that $z_\eps(t)$ belongs to $\Omega$. For this reason, we introduce 
\begin{equation}\label{def:maxi-T'}
 T_{\eps}'=\sup\left\{T_0\in [0,T_{\varepsilon}] \, \,: \, z_{\varepsilon}(t)\in \cal C_{\rho_b} \quad \text{for all} \, t\in[0,T_0)\right\}.
\end{equation}
One of the purposes of this section is to prove that $ T_{\eps}'=T_{\varepsilon}$.

We start with the following lemma.
\begin{lemma}\label{lemma:b(z)}
Under Assumption~\ref{ass: main part}, there exists $C>0$ such that for all $t\in [0,T_\eps')$
\begin{equation*}
|b(z_\eps(t))-b(z^0)|\leq C t \left(\sup_{s\in [0,t]}\sqrt{\Ie(s)}+\frac{1}{\lne}\right)+C\eps.
\end{equation*}
\end{lemma}

\begin{proof}We use the weak formulation of \eqref{eq:transport-1} given by Proposition~\ref{prop:well-posed} with test function $\Phi(x)=x$. We therefore have in $\cal{D}'(0,T_{\varepsilon})$ and by the decomposition \eqref{eq:decomposition} of the velocity field $\veps$:
\begin{align*}
\frac{\dd}{\dd t} b(z_\eps)=&\dot{z}_\eps\cdot \nabla b(z_\eps)\\
=&\frac{1}{\gamma_{\varepsilon}\lne} \nabla b(z_\eps)\cdot \int_{\Omega} (v_{\varepsilon}+F_{\varepsilon})(x)(b\weps)(x)\,\dd x \\
=&\frac{1}{\gamma_{\varepsilon}\lne} \nabla b(z_\eps)\cdot \int_{\Omega} v_{\varepsilon,K}(x)(b\weps)(x)\,\dd x \\
&+\frac{1}{\gamma_{\varepsilon}\lne} \nabla b(z_\eps)\cdot \int_{\Omega} v_{\varepsilon,L}(x)(b\weps)(x)\,\dd x \\
&+\frac{1}{\gamma_{\varepsilon}\lne} \nabla b(z_\eps)\cdot \int_{\Omega} (v_{\varepsilon,R}+F_{\varepsilon})(x)(b\weps)(x)\,\dd x.
\end{align*}

We estimate the first term by the skew-symmetry of $K$
\begin{align*}
4\pi\Big| & \int_{\Omega} v_{\varepsilon,K}(x)(b\weps)(x)\,\dd x \Big|\\
=&
2\Big| \iint_{\Omega^2} \frac{1}{ b(x)} K(x,y)\sqrt{b(x)b(y)}(b\weps)(x)(b\weps)(y) \,\dd x\dd y\Big|\\
=&
\Big| \iint_{\Omega^2} \Big(\frac{1}{ b(x)}-\frac{1}{ b(y)}\Big) K(x,y)\sqrt{b(x)b(y)}(b\weps)(x)(b\weps)(y) \,\dd x\dd y\Big|\\
\leq& C
\end{align*}
 where we have used as usual the mean value theorem \eqref{MVT}.
 
 For the second term, we compute
\begin{align*}
4\pi \Big|\nabla b(z_\eps)\cdot \int_{\Omega} v_{\varepsilon,L}(x)&(b\weps)(x)\,\dd x \Big|
=\Big| \nabla b(z_\eps)\cdot \int_{\Omega} \frac{\nabla^\perp b(x)}{b^2(x)}\psi_{\eps}(x)(b\weps)(x)\,\dd x \Big|\\
&=\Big| \nabla b(z_\eps)\cdot \int_{\Omega} \frac{\nabla^\perp b(x)-\nabla^\perp b(z_{\varepsilon})}{b^2(x)} \psi_{\eps}(x)(b\weps)(x)\,\dd x \Big| \\
&\leq C\lne \int_{\Omega}|x-z_\eps|(b\weps)(x)\,\dd x \leq C \lne \sqrt{\Ie},
\end{align*}
where we have used Lemma~\ref{lem:est-psi}, the mean value inequality \eqref{MVT} and that $z_{\varepsilon}\in \cal C_{\rho_b}\subset \Omega_{r_{0}}$ for $t\in [0,T_{\varepsilon}')$. We have also used that $\supp (\omega_\eps(t,\cdot))\subset \Omega_{r_0}$ on $[0,T_\eps)$ with $\inf_{\Omega_{r_0}} b>0$.

The bound for last term follows immediately from \eqref{bound:vR} and the properties of $F_{\eps}$, see Assumption~\ref{ass: F},
\[
\Big|\int_{\Omega} (v_{\varepsilon,R}+F_{\varepsilon})(x)(b\weps)(x)\,\dd x \Big|\leq C
\]
Combining the previous three estimates, we infer that
\[
\Big|\frac{\dd}{\dd t} b(z_\eps(t)) \Big| \leq C\Big( \sqrt{\Ie} +\frac1\lne\Big) \quad \forall t\in [0,T_{\varepsilon}').
\]

Finally, we observe that 
$$
z_\eps(0)=\frac{1}{\gamma_{\varepsilon}}\int_{\Omega} x (b\omega_\eps^0)(x)\,\dd x=z^0+\frac{1}{\gamma_{\varepsilon}}\int_{\Omega} (x-z^0) (b\omega_\eps^0)(x)\,\dd x
$$
 therefore $|b(z_\eps(0))-b(z^0)|\leq C |z_\eps(0)-z^0|\leq C\eps $ by virtue of \eqref{MVT}. The conclusion follows.
\end{proof}

Next, we derive an estimate on the time derivative of the momentum $\Ie$ as defined in \eqref{def:momentum}.
\begin{proposition}\label{prop:momentum} 
Under Assumption~\ref{ass: main part}, there exists $C>0$ such that for all $t\in [0,T_\eps')$
\begin{equation*}
\frac{\dd}{\dd t} \Ie \leq C\Ie +\frac{C}{\lne}.
\end{equation*}
\end{proposition}

\begin{proof}Using the weak formulation of \eqref{eq:transport-1} in Proposition~\ref{prop:well-posed}, we compute the time derivative of $\Ie$ defined in \eqref{def:momentum}
\begin{align*}
\frac{\dd}{\dd t}\Ie(t) =&2\int_{\Omega} (x-z_\eps(t))\cdot \left(\frac{(v_\eps+F_{\varepsilon})(x)}{\lne}-\dot{z}_\eps(t)\right)(b\weps)(x)\,\dd x\\
=&\frac2\lne\int_{\Omega} (x-z_\eps)\cdot v_{\varepsilon,K}(x) (b\weps)(x)\,\dd x\\
&+\frac2\lne\int_{\Omega} (x-z_\eps)\cdot v_{\varepsilon,L}(x) (b\weps)(x)\,\dd x\\
&+\frac2\lne\int_{\Omega} (x-z_\eps)\cdot (v_{\varepsilon,R}+F_{\varepsilon})(x) (b\weps)(x)\,\dd x\\
&-2 \dot{z}_\eps \cdot \int_{\Omega} (x-z_\eps) (b\weps)(x)\,\dd x,
\end{align*}
where we have decomposed $\veps$ by means of \eqref{eq:decomposition}.

We treat the first term as in Lemma~\ref{lemma:b(z)}, namely by skew-symmetry of $K$, to obtain 
\begin{align*}
4\pi \int_{\Omega} &(x-z_\eps)\cdot v_{\varepsilon,K}(x) (b\weps)(x)\,\dd x \\
=&2\iint_{\Omega^2} (x-z_\eps)\cdot \frac{1}{ b(x)} K(x,y)\sqrt{b(x)b(y)}(b\weps)(x)(b\weps)(y) \,\dd x\dd y\\
=&\iint_{\Omega^2} (x-z_\eps)\cdot \frac{1}{ b(x)} K(x,y)\sqrt{b(x)b(y)}(b\weps)(x)(b\weps)(y) \,\dd x\dd y\\
&-\iint_{\Omega^2} (y-z_\eps)\cdot \frac{1}{ b(y)} K(x,y)\sqrt{b(x)b(y)}(b\weps)(x)(b\weps)(y) \,\dd x\dd y\\
=& \iint_{\Omega^2} (x-y)\cdot \frac{1}{ b(x)} K(x,y)\sqrt{b(x)b(y)}(b\weps)(x)(b\weps)(y) \,\dd x\dd y\\
&-\iint_{\Omega^2} (y-z_\eps)\cdot \Big(\frac{1}{ b(y)}-\frac{1}{ b(x)} \Big)K(x,y)\\
&\hspace{4cm}\times \sqrt{b(x)b(y)}(b\weps)(x)(b\weps)(y) \,\dd x\dd y\\ 
\end{align*}
which is bounded by a constant because of the mean value theorem \eqref{MVT} and $x,y\in \cal C_{\rho_b}$ for all $x,y\in \supp(\weps(t,\cdot))$ with $t\in [0,T_{\eps}).$

We treat the second term at the end of this proof. The bound for the third term easily follows from \eqref{bound:vR} and Assumption~\ref{ass: F}:
\[
\Big|\int_{\Omega} (x-z_\eps)\cdot (v_{\varepsilon,R}+F_{\varepsilon})(x) (b\weps)(x)\,\dd x \Big|\leq C.
\]

Following the proof of Lemma~\ref{lemma:b(z)} we compute $\dot{z}_\eps(t)$ by
\begin{align*}
\lne \dot{z}_\eps(t)=&\frac{1}{\gamma_{\varepsilon}} \int_{\Omega} (v_{\varepsilon}+F_{\varepsilon})(x)(b\weps)(x)\,\dd x \\
=&\frac{1}{\gamma_{\varepsilon}} \int_{\Omega} v_{\varepsilon,K}(x)(b\weps)(x)\,\dd x \\
&+\frac{1}{\gamma_{\varepsilon}} \int_{\Omega} v_{\varepsilon,L}(x)(b\weps)(x)\,\dd x \\
&+\frac{1}{\gamma_{\varepsilon}} \int_{\Omega} (v_{\varepsilon,R}+F_{\varepsilon})(x)(b\weps)(x)\,\dd x.
\end{align*}
The first and third term on the right-hand-side are bounded by a constant as in the proof of Lemma~\ref{lemma:b(z)}, namely by exploiting the skew-symmetry of $K$, the distance $r_{0}$ in Remark~\ref{rem:r0}, \eqref{bound:vR} for $v_{\varepsilon,R}$ and Assumption~\ref{ass: F} for $F_{\varepsilon}$. It follows
\begin{equation}\label{eq:derivee-z}
\lne \dot{z}_\eps(t)=\frac{1}{\gamma_{\varepsilon}}\int_{\Omega} v_{\eps,L}(x)(b\weps)(x)\,\dd x + \mathcal{O}(1).
\end{equation}

Combining this equality with the two estimates above, we get
\begin{align*}
\frac{\dd}{\dd t}\Ie(t) =&\frac2\lne\int_{\Omega} (x-z_\eps)\cdot v_{\varepsilon,L}(x) (b\weps)(x)\,\dd x +\mathcal{O}\Big(\frac1\lne\Big)\\
&-\frac2\lne \Big(\frac{1}{\gamma_{\varepsilon}}\int_{\Omega} v_{\eps,L}(x)(b\weps)(x)\,\dd x \Big) \cdot \Big(\int_{\Omega} (x-z_\eps) (b\weps)(x)\,\dd x\Big)\\
=&\frac2{\gamma_{\varepsilon}\lne} \iint_{\Omega^2} ( v_{\eps,L}(x)- v_{\eps,L}(y)) \cdot (x-z_\eps) (b\weps)(x)(b\weps)(y)\,\dd x\dd y \\
&+\mathcal{O}\Big(\frac1\lne\Big)\\
=&\frac1{2\pi\gamma_{\varepsilon}\lne}\iint_{\Omega^2} \left(\frac{\nabla^\perp b(x)}{ b^2(x)}\psi_\eps(x)-\frac{\nabla^\perp b(y)}{b^2(y)}\psi_\eps(y)\right) \cdot (x-z_\eps) \\
&\hspace{4cm} \times(b\weps)(y) (b\weps)(x)\,\dd x\dd y +\mathcal{O}\Big(\frac1\lne\Big).
\end{align*}
Introducing $b(z_\eps)$ we rewrite this formula as follows
\begin{align*}
2\pi\gamma_{\varepsilon}&\lne\frac{\dd}{\dd t}\Ie(t) \\
=&\gamma_{\varepsilon}\int_{\Omega} \left(\frac{\nabla^\perp b(x)}{ b^2(x)}-\frac{\nabla^\perp b(z_\eps)}{b^2(z_\eps)}\right) \cdot (x-z_\eps)\psi_\eps(x)(b\weps)(x)\,\dd x\\
&+\gamma_{\varepsilon}\int_{\Omega} \frac{\nabla^\perp b(z_\eps)}{b^2(z_\eps)} \cdot (x-z_\eps)\psi_\eps(x)(b\weps)(x)\,\dd x\\
&-\iint_{\Omega^2} \left(\frac{\nabla^\perp b(y)}{b^2(y)}-\frac{\nabla^\perp b(z_{\varepsilon})}{ b^2(z_{\varepsilon})} \right) \cdot (x-z_\eps)\psi_\eps(y)(b\weps)(y) (b\weps)(x)\,\dd x\dd y \\
&-\iint_{\Omega^2} \frac{\nabla^\perp b(z_{\varepsilon})}{ b^2(z_{\varepsilon})} \cdot (x-z_\eps)\psi_\eps(y)(b\weps)(y) (b\weps)(x)\,\dd x\dd y +\cal{O}(1).
\end{align*}
Recalling Lemma~\ref{lem:est-psi} for $\psi_\eps$ and that $\supp(\weps(t,\cdot))\subset \cal C_{\rho_b}\subset \Omega_{r_0}$ for any $t\in [0,T_{\eps})$ by Assumption~\ref{ass: main part} and $z_{\eps}(t)\in \Omega_{r_0}$ for all $t\in[0,T_{\eps}')$, we have 
\begin{align*}
&\gamma_{\varepsilon}\Big| \int_{\Omega} \left(\frac{\nabla^\perp b(x)}{ b^2(x)}-\frac{\nabla^\perp b(z_\eps)}{b^2(z_\eps)}\right) \cdot (x-z_\eps)\psi_\eps(x)(b\weps)(x)\,\dd x\Big| \\
&+ \Big| \iint_{\Omega^2} \left(\frac{\nabla^\perp b(y)}{b^2(y)}-\frac{\nabla^\perp b(z_{\varepsilon})}{ b^2(z_{\varepsilon})} \right) \cdot (x-z_\eps)\psi_\eps(y)(b\weps)(y) (b\weps)(x)\,\dd x\dd y \Big| \\
&\leq C \|\nabla b^{-1}\|_{L^\infty(\Omega_{r_{0}/2})}\Big( \int_{\Omega} |x-z_\eps|^2 |\psi_\eps(x)|
 (b\weps)(x)\,\dd x\\
 &\hspace{3cm}+ \iint_{\Omega^2} |x-z_\eps| |y-z_\eps| |\psi_\eps(y)| (b\weps)(y) (b\weps)(x)\,\dd x\dd y \Big)\\
&\leq C \lne \Ie.
 \end{align*}
For the other terms, we invoke Lemma~\ref{lemma:L2} to conclude
\begin{align*}
\Bigg|\int_{\Omega}&\frac{\nabla^\perp b(z_\eps)}{b^2(z_\eps)} \cdot (x-z_\eps) \Big( \gamma_{\varepsilon} \psi_\eps(x) 
 -\int_{\Omega} \psi_\eps(y)(b\weps)(y)\,\dd y \Big) (b\weps)(x)\,\dd x \Bigg|\\
&\leq C\sqrt{\Ie} \left(\int_{\Omega}\left| \gamma_{\varepsilon}\psi_\eps(x)-\int_{\Omega}\psi_\eps(y)(b\weps)(y)\,\dd y\right|^2(b\omega_\eps)(x) \,\dd x\right)^{1/2}\\
&\leq C\sqrt{\Ie} \lne^{1/2} \leq C\lne \Ie(t) +C .
\end{align*}

We have readily found
\[
\frac{\dd}{\dd t}\Ie(t) \leq C\Ie(t) +\frac{C}\lne
\]
which ends the proof.
\end{proof}

This finally enables us to prove the following estimate on the moment of inertia.
\begin{proposition}\label{prop:momentum-bis} 
Under Assumption~\ref{ass: main part}, there exists $\varepsilon_{0}>0$ depending only on $T$, $b$, $\Omega$, $(z_i^0)_{i=1,\dots, {N_v}}$, $(\gamma_i)_{i=1,\dots, {N_v}}$, $(\Gamma^i)_{i=1,\dots, {N_{is}}}$, $M_0$, such that for $\eps \leq \eps_0$ it holds $T'_\eps=T_\eps$ with $T_{\eps}$ and $T_{\eps}'$ defined in \eqref{def:maxi-T} and \eqref{def:maxi-T'} respectively.
Moreover, there exists $C>0$ such that 
\begin{equation*}
|b(z_\eps(t))-b(z^0)|^2+\Ie(t)\leq \frac{C}{\lne},\quad \forall t\in [0,T_\eps).
\end{equation*}
\end{proposition}

\begin{proof}
We have proved in Proposition~\ref{prop:momentum}, that for $t\in [0,T_\eps')$
\[
\frac{\dd}{\dd t} \Ie(t)\leq C \Ie +\frac{C}{\lne}
\]
and recalling that $\Ie(0)\leq C \eps^2$ we obtain 
\begin{equation*}
\Ie(t)\leq \frac{C_0}{\lne},\quad \forall t\in [0,T_\eps')
\end{equation*}
 for some constant $C_0$ depending only on $T$, where $T_\eps'\leq T_{\eps}\leq T$.

By Lemma~\ref{lemma:b(z)}, we get for $t\in [0,T_\eps')$ that
\begin{equation*}
|b(z_\eps(t))-b(z^0)| \leq \frac{C}{\sqrt{\lne}}.
\end{equation*}
Consequently, by the definition of $\cal C_{\rho_b}$, see Remark~\ref{rem:r0}, together with the definition of $T'_\eps$, see \eqref{def:maxi-T'}, we conclude that for $\varepsilon_{0}$ small enough, this inequality implies that $z_\eps(t)\in \cal C_{\rho_b/2}$ for all $t\in [0,T_{\eps}')$, which implies that $T'_\eps=T_\eps$ and the desired estimate holds on $[0,T_\eps)$.
\end{proof}

The first important consequence of Proposition~\ref{prop:momentum-bis} is the following weak localization property.
\begin{proposition}\label{prop:weak-loc}
Under Assumption~\ref{ass: main part}, there exists $C>0$ such that for all $\eps \leq \eps_0$ (defined in Proposition~\ref{prop:momentum-bis}), we have for all $t\in [0,T_\eps)$
\[
\int_{\Omega\setminus B(z_\eps(t), R_\eps)} (b\omega_\eps)(t,x)\,\dd x\leq \frac{C}{\ln \lne},
\]
where 
$$
R_\eps=\left(\frac{\ln \lne}{\lne}\right)^{1/2}.
$$
\end{proposition}

\begin{proof}
We simply write
\[
\int_{\Omega\setminus B(z_\eps(t), R_\eps)} (b\omega_\eps)(t,x)\,\dd x\leq \frac{1}{R_\eps^2}\Ie(t),
\]
and the conclusion follows from Proposition~\ref{prop:momentum-bis}.
\end{proof}

\begin{remark}\label{rem:momemtum}
It follows from Proposition~\ref{prop:weak-loc} that the vorticity $\weps(t)$ remains weakly localized close to the center of vorticity $z_\eps(t)$ up to time $T_{\eps}$ defined in \eqref{def:maxi-T}. In order to prove the weak localization property stated in Theorem~\ref{thm: main red system}, it remains to identify the limiting trajectory and to show that $T_{\eps}=T$. We refer to Section~\ref{sec:trajectory} for the former item, the latter requires a strong localization of $\weps(t)$, namely that the support of $\weps(t)$ is sharply localized. While the initial data satisfy $\supp(\weps^0)\subset B(z^0,\eps M_0)$ such a property is not expected to hold for positive times, see the introduction for the filamentation phenomenon. The term $v_{\eps,L}$ defined in \eqref{eq:uL} which accounts for the limiting dynamics of the point vortices lacks suitable regularity properties and does not allow for an adaptation of the classical method \cite{Marchioro-Pulvirenti-book, Marchioro-Pulvirenti} developed for point vortices for 2D Euler equations. For the axisymmetric 3D Euler equations without swirl, which can be seen as a special case of \eqref{eq:lake} with $b(z,r)=r$, the authors of \cite{Mar2} were able to obtain the result only up to short times, due to this difficulty. In \cite{Mar3}, this obstacle has been overcome by arguing that a strong localization property in the radial direction suffices to show that $T_{\eps}=T$. 

For the lake equations, this translates to seeking a strong localization property in the direction of the steepest ascent of $b$, namely in the direction of $\nabla b$ which motivates the level-set approach chosen in this paper. We will refer to this direction as transverse direction. Note that $v_{\eps,L}$ is orthogonal to that direction. In view of the definition \eqref{def:maxi-T} of $T_{\eps}$, this in turn will enable us to infer $T_{\eps}=T$, see Proposition~\ref{prop:strong-loc}.

However, the method presented here differs from the one developed in \cite{Mar2, Mar3} for the axisymmetric 3D Euler equations without swirl in several aspects. We need to deal with additional difficulties arising close to the boundary due to both the geometry of $\Omega$ and the degeneracy of $b$. Further, we do not rely on an explicit Biot-Savart law type formula which is not available for \eqref{eq:divcurl} but extensively used in \cite{Mar2, Mar3}. Moreover, we have provided a simplified proof of the weak localization property that in particular does not require bounds on second order moments as in \cite[Lemma 4.1]{Mar3}. 

The strong localization property of the vorticity in the transverse direction is the main objective of Section~\ref{sec:transverse}. We refer the reader to Remark~\ref{rem:strong-loc} for a comparison of the proofs of strong localization property in the present paper and \cite{Mar3}.
 \end{remark}

\section{Strong localization in the transverse direction}\label{sec:transverse}

As in the previous sections, for $T>0$ and under Assumption~\ref{ass: main part} we denote by $(\omega_\eps, v_\eps)$ the unique weak solution of \eqref{eq:transport-1} in the sense of Proposition~\ref{prop:well-posed} on $[0,T]$.

\medskip

The aim of the present section consists in proving a strong localization property in the transverse direction. This in turn will enable us to state a global result, i.e. that $T_{\eps}=T$ for $T>0$ arbitrarily large, which could not be inferred from the weak localization property alone, see Remark~\ref{rem:momemtum}.
To that end, we introduce the moment of inertia in the transverse direction 
\[
\mathcal{K}_\eps(t)=\int_{\Omega}| b(x)-b(z^0)|^2 b(x)\omega_\eps(t,x)\,\dd x.
\]
Thanks to Lemma~\ref{lemma:J} we immediately obtain
\begin{equation}\label{eq:est-K}
\mathcal{K}_\eps(t)\leq \frac{C}{\lne}, \quad \forall t\in [0,T_\eps)
\end{equation}
by developing the square.

To prove the strong localization property, we follow the strategy developed by {\sc C. Marchioro} and {\sc M. Pulvirenti} \cite{Marchioro-Pulvirenti-book} in the planar case and subsequently adapted to the 3D axisymmetric Euler equation without swirl see e.g. \cite{Mar2,Mar3}. The main difference to the original proof in \cite{Marchioro-Pulvirenti-book} is that all the quantities are defined in the transverse direction only, see Remark~\ref{rem:momemtum}. This localization in one direction of the respective quantities is the key ingredient in \cite{Mar3} which allows the authors to upgrade the local-in-time result of \cite{Mar2} to a global result \cite{Mar3}. 

Indeed, when arbitrary directions are considered, the term $v_{\varepsilon,L}$ defined in \eqref{eq:uL} displays a diverging (in $\eps$) term that lacks a suitable control. The projection on the transverse direction avoids the presence of $v_{\varepsilon,L}$.
We start by adapting the estimates in the radial direction in \cite{Mar3} to obtain a strong localization in the transverse direction for $\weps$. 

\begin{lemma}\label{lemma:Rt}
 Under Assumption~\ref{ass: main part}, we define
 \[
 R_{t}:=\max \left\{ | b(x) - b(z^0) |\:| x\in \supp \omega_{\varepsilon}(t,\cdot) \right\}.
 \]
 For any $t\in (0,T_{\varepsilon}]$, we consider $x_{0} \in \supp \omega_{\varepsilon}^0$ such that at time $t$, we have
 \[
 |b(X_{\varepsilon}(t,x_{0})) - b(z^0) |=R_{t}.
 \]
 Then, at this time $t$,
 \[
 \frac{\dd}{\dd t} | b(X_{\varepsilon}(t,x_{0})) - b(z^0) | \leq \frac{C}{|\ln \varepsilon|}+ \frac{C}{R_{t}|\ln \varepsilon|} + \frac{C \sqrt{m_{t}(R_{t}/2)}}{\varepsilon |\ln\varepsilon|},
 \]
 where the function $m_{t}(\cdot)$ is defined by
 \[
 m_{t}(h):= \int_{A_{h}^c} (b\omega_{\varepsilon})(t,y)\,\dd y \text{ with }A_{h}:=\{y\in \Omega\, :\, |b(y) - b(z^0) |\leq h\}.
 \]
\end{lemma}

\begin{remark}\label{rem:strong-loc}
While our method is inspired by the one in \cite{Mar2,Mar3}, we choose to consider simplified definitions of $R_{t}$ and $A_{h}$ compared to \cite{Mar3}. More precisely, the equivalent definitions to \cite{Mar3} would have led to consider $b(z_{\varepsilon}(t))$ instead to $b(z^{0})$ in the respective definitions of $R_{t}$ and $A_{h}$. However, the main objective is to prove that $\supp(\weps)$ is localized close to $\mathcal{C}$ being the connected component of $\{b(x)=b(z^0)\}$ containing $z^0$, see {\it (i)} Theorem~\ref{thm: main red system}. Considering these quantities turns out to be sufficient and simplifies the computation of the corresponding time derivatives.
\end{remark}

\begin{proof} Let $x\in \supp \omega_{\varepsilon}(t,\cdot)$ such that $R_t=| b(x) - b(z^0) |$. 
By Proposition~\ref{prop:transport} there exists $x_0=X_{\varepsilon}(0,t,x)\in \supp \omega_{\varepsilon}^0$ such that $x=X_{\varepsilon}(t,0,x_0)=X_{\varepsilon}(t,x_{0})$. 
We compute the derivative with respect to time for $s\in [0,t]$:
\begin{align*}
 \frac{\dd}{\dd s} | b(X_{\varepsilon}&(s,x_{0})) - b(z^0)) | \\
 = & \frac1{|\ln \varepsilon|}\Big(\nabla b(X_\eps(s,x_0)) \cdot (v_{\varepsilon}+F_{\varepsilon})(s,X_\eps(s,x_0)) \Big) \frac{ b(X_\eps(s,x_0)) - b(z^0) }{ | b(X_\eps(s,x_0)) - b(z^0) | }\\
 \leq& 
 \frac1{|\ln \varepsilon|} | \nabla b(X_\eps(s,x_0)) \cdot v_{\varepsilon,K}(s,X_\eps(s,x_0)) | 
 \\
 &+
 \frac1{|\ln \varepsilon|}| 
\nabla b(X_\eps(s,x_0)) \cdot (v_{\varepsilon,R}+F_{\varepsilon})(s,X_\eps(s,x_0))|,
\end{align*}
where the special form of $v_{\varepsilon,L}$ \eqref{eq:uL}, namely the orthogonality of $\nabla b$ and $v_{\varepsilon,L}$, is crucially used.

By \eqref{bound:vR} and Assumption~\ref{ass: F}, the second term in the right hand side at time $t$ is bounded by $C/|\ln \varepsilon|$ (since $X_\eps(s,x_0)\in \supp \omega_\eps(s,\cdot)\subset \cal C_{\rho_b}$ implies that $\nabla b(X_\eps(s,x_0))$ is bounded).

For the first term, we split the integral defining $v_{\varepsilon,K}(t,x)$ \eqref{eq:uK} on the subset $A_{R_{t}/2}$ and on the complementary of $A_{R_t/2}$ where we recall $A_{h}:=\{ |b(y) - b(z^0) |\leq h\}$. 

\medskip

On $A_{R_{t}/2}$, we use the mean value theorem \eqref{MVT} 
\[
\frac{R_{t}}2 \leq |b(x)-b(y)| \leq a\|\nabla b\|_{L^\infty(\Omega_{r_{0}/2})} |x-y| 
\] 
to get
\[
\Bigg| \frac{\nabla b(x)}{b(x)} \cdot \int_{A_{R_{t}/2}} K(x,y)\sqrt{b(x)b(y)} (b\omega_{\varepsilon})(y) \, dy \Bigg|\leq \frac C{R_{t}}.
\]
On $A_{R_{t}/2}^c$, we have 
\[
\Bigg| \frac{\nabla b(x)}{b(x)} \cdot \int_{A_{R_{t}/2}^c} K(x,y) \sqrt{b(x)b(y)} (b\omega_{\varepsilon})(y) \, dy \Bigg|
\leq C\int_{A_{R_{t}/2}^c} \frac{(b\omega_{\varepsilon})(y)}{|x-y|}\, dy
\]
 and we estimate the right-hand side using the well-known fact, see e.g. \cite{IftimieJEDP}, that for all $h\in L^1\cap L^\infty(S)$ we have $\int_S |h(y)|/|x-y|\,dy \leq C \|h\|_{L^1(S)}^{1/2} \|h\|_{L^\infty(S)}^{1/2}$:
\begin{equation*}
\int_{A_{R_{t}/2}^c} \frac{(b\omega_{\varepsilon})(y)}{|x-y|}\, dy\leq \frac{C}{\varepsilon}\sqrt{m_t(R_t/2)}.
 \end{equation*}
 This concludes the proof of the lemma.
\end{proof}

The second step aims towards the strong localization property consists in proving that for any $\ell>0$, the mass of vorticity outside $A_{1/|\ln\varepsilon|^k}$ is smaller than $\varepsilon^\ell$ provided $\varepsilon$ sufficiently small.

\begin{lemma} \label{lemma-step}For any $\ell>0$ and $k\in (0,1/4)$, we have
\[
\lim_{\varepsilon \to 0} \varepsilon^{-\ell} m_{t}\Big(\frac1{|\ln\varepsilon|^k} \Big) =0.
\]
\end{lemma}

For the particular choice $b(z,r)=r$, namely for the axisymmetric 3D Euler equations without swirl, Lemma~\ref{lemma-step} recovers the statement of \cite[Lemma 3.4]{Mar3}.

\begin{proof}
We adapt again the proof of \cite[Lemma 3.4]{Mar3} to our problem, namely \eqref{eq:transport-1}. We begin by defining a mollified version of $m_{t}$: 
 \[
 \mu_{t}(R,h) = \int_{\Omega} \Big(1-W_{R,h}\Big(b(y) - b(z^0)\Big)\Big) (b\omega_{\varepsilon})(t,y)\,\dd y
 \]
 where $W_{R,h}$ is a smooth non-negative function from $\R$ to $\R$ such that
 \[
 W_{R,h}(s) =
\begin{cases}
 1 &\text{if }|s|\leq R\\
 0 &\text{if }|s|\geq R+h
\end{cases}
 \]
 with $W_{R,h}^{(p)}(s) \leq C_{p}/h^{p}$ for $p=1,2$. The function $W_{R,h}$ satisfies
 \[
 \mu_{t} (R,h) \leq m_{t}(R) \leq \mu_{t} (R-h,h) .
 \]
 To prove Lemma~\ref{lemma-step}, it then suffices to find a sequence $(h_{\varepsilon})_{\varepsilon}$ such that
 \[
 \varepsilon^{-\ell} \mu_{t}\Big(\frac1{|\ln\varepsilon|^k} -h_{\varepsilon} ,h_{\varepsilon} \Big) \to 0.
 \]
 The smooth function $\mu_t$ is differentiable with respect to time, and we compute
\begin{align*}
\frac{\dd}{\dd t} \mu_{t}(R,h) =&
 -\frac1{|\ln\varepsilon|}\int \nabla \Big(W_{R,h}\Big(b(y) - b(z^0)\Big)\Big)\cdot (v_{\varepsilon}+F_{\varepsilon})(y)(b\omega_{\varepsilon})(y)\,\dd y\\
=&- \int_{\Omega} W_{R,h}'\Big(b(y) - b(z^0)\Big) \frac1{|\ln\varepsilon|}\nabla b(y)\cdot (v_{\varepsilon}+F_{\varepsilon})(y) (b\omega_{\varepsilon})(y)\,\dd y.
\end{align*}
Upon applying the decomposition \eqref{eq:decomposition} for $\veps$, we use the bound \eqref{bound:vR} of $v_{\varepsilon,R}$ and Assumption~\ref{ass: F} for $F_{\varepsilon}$ to get
\[
 \Big| \int_{\Omega} W_{R,h}'\Big(b(y) - b(z^0)\Big)\frac1{|\ln\varepsilon|}\nabla b(y)\cdot (v_{\varepsilon,R}+F_{\varepsilon})(y) (b\omega_{\varepsilon})(y)\,\dd y \Big| \leq \frac{Cm_{t}(R)}{h|\ln \varepsilon |}.
\]
For the term containing $v_{\eps,K}$, we write that
\begin{align*}
\int_{\Omega} & W_{R,h}'\Big(b(y) - b(z^0)\Big) \nabla b(y)\cdot v_{\varepsilon,K}(y) (b\omega_{\varepsilon})(y)\,\dd y \\
=&\frac1{2\pi}\iint_{\Omega^2}W_{R,h}'\Big(b(y) - b(z^0)\Big) \frac{\nabla b(y)}{b(y)}\cdot K(y,x) \sqrt{b(x)b(y)} (b\omega_{\varepsilon})(y)(b\omega_{\varepsilon})(x)\,\dd y \dd x\\
=& \frac1{4\pi}\iint_{\Omega^2} f(t,x,y)\,\dd y \dd x
\end{align*}
 where 
\begin{multline*}
f(t,x,y):= \Bigg(W_{R,h}' \Big(b(y) - b(z^0)\Big) \frac{\nabla b(y)}{b(y)} - W_{R,h}'\Big(b(x) - b(z^0)\Big) \frac{\nabla b(x)}{b(x)}\Bigg)\\
\cdot K(y,x) \sqrt{b(x)b(y)} (b\omega_{\varepsilon})(y)(b\omega_{\varepsilon})(x)
\end{multline*}
vanishes if $x,y\in A_{R}$ due to the support properties of $W_{R,h}'$. Hence,
\begin{align*}
 \iint_{\Omega^2}& f(t,x,y)\,\dd y \dd x \\
 =& \iint_{x\in A_{R}^c} f(t,x,y)\,\dd y\dd x+\iint_{y\in A_{R}^c} f(t,x,y)\,\dd y \, dx - \iint_{x,y\in A_{R}^c} f(t,x,y)\,\dd y\dd x\\
 =& 2\iint_{x\in A_{R}^c} f(t,x,y)\,\dd y\dd x - \iint_{x,y\in A_{R}^c} f(t,x,y)\,\dd y\dd x\\
 =& 2\iint_{(x,y)\in A_{R}^c\times A_{R-h^\alpha}} f(t,x,y)\,\dd y\dd x+2\iint_{(x,y)\in A_{R}^c\times A_{R-h^\alpha}^c} f(t,x,y)\,\dd y\dd x \\
 &- \iint_{x,y\in A_{R}^c} f(t,x,y)\,\dd y\dd x
\end{align*}
where $\alpha>0$ is a parameter to be fixed later such that $R>2 h^\alpha$. For all $(x,y)\in A_{R}^c\times A_{R-h^\alpha}$ we have by the mean value theorem \eqref{MVT} that
\[
h^\alpha \leq|b(x)-b(z^0)| - |b(y)-b(z^0)| \leq|b(x)-b(y)| \leq C_{r_{0}} |x-y|.
 \]
It follows that
\begin{align*}
 \Big|\iint_{(x,y)\in A_{R}^c\times A_{R-h^\alpha}} f(t,x,y)\,\dd y\dd x \Big| \leq C\frac{m_{t}(R)}{h^{1+\alpha}}.
\end{align*}
For the remaining two integrals, we use that 
\begin{multline*}
\Bigg|W_{R,h}' \Big(b(y) - b(z^0)\Big) \frac{\nabla b(y)}{b(y)} - W_{R,h}'\Big(b(x) - b(z^0)\Big) \frac{\nabla b(x)}{b(x)} \Bigg| \\
\leq C\Big(\frac{1}{h^2}+\frac{1}{h}\Big) |x-y|.
\end{multline*}
Therefore, provided that $h\leq 1$ and $R>2 h^\alpha$, implying $R-h^\alpha >R/2$, we have
\begin{align*}
 2\Bigg|\iint_{(x,y)\in A_{R}^c\times A_{R-h^\alpha}^c}& f(t,x,y)\,\dd y\dd x\Bigg| + \Bigg|\iint_{x,y\in A_{R}^c} f(t,x,y)\,\dd y\dd x\Bigg|\\
 \leq&\frac{C}{h^2}\iint_{(x,y)\in A_{R}^c\times A_{R/2}^c} (b^2\omega_{\varepsilon})(y)(b\omega_{\varepsilon})(x)\, \dd x \dd y\\
 \leq&\frac{C}{h^2 R^2} \mathcal{K}_{\varepsilon}(t) m_{t}(R) \leq \frac{Cm_{t}(R)}{h^2 R^2 |\ln \varepsilon|},
\end{align*}
where we have used \eqref{eq:est-K} in the last inequality.
Combining these estimates, we obtain that
\begin{equation}\label{Mar3:eq1}
 \frac{\dd}{\dd t} \mu_{t}(R,h) \leq A_{\varepsilon}(R,h) m_{t}(R)
\end{equation}
with
\begin{equation}\label{Mar3:eq2}
A_{\varepsilon}(R,h):= C\Big(\frac{1}{h|\ln \varepsilon |} + \frac{1}{h^{1+\alpha} |\ln \varepsilon|}+\frac{1}{h^2 R^2 |\ln \varepsilon|^2}\Big).
\end{equation}
We note that the estimate on $A_{\varepsilon}(R,h)$ coincides with the one obtained in \cite[Equation (3.44)]{Mar3}. It hence suffices to reproduce verbatim the remaining part of the proof of \cite[Lemma 3.4]{Mar3} in order to complete the proof of Lemma~\ref{lemma-step}. Indeed, the rest of the proof of \cite[Lemma 3.4]{Mar3} is an iterative argument only based on \eqref{Mar3:eq1}-\eqref{Mar3:eq2}.
\end{proof}

We finish this section with the strong localization property, namely {\it(i)} of Theorem~\ref{thm: main red system}.

\begin{proposition}\label{prop:strong-loc}
Under Assumption~\ref{ass: main part}, there exists $\varepsilon_{0}$, depending only on $T$, $b$, $\Omega$, $(z_i^0)_{i=1,\dots, {N_v}}$, $(\gamma_i)_{i=1,\dots, {N_v}}$, $(\Gamma^i)_{i=1,\dots, {N_{is}}}$, $M_0$, such that for $\eps \leq \eps_0$, for $T_{\eps}$ defined in \eqref{def:maxi-T} it holds $T_{\eps}=T$ and the vorticity $\weps$ is strongly localized in the direction of steepest ascent of $b$, namely for every $k\in (0,1/4)$ there exists $\eps_{k,T}, C_{k,T}>0$ depending only on $z^0$, $\gamma$, $M_{0}$, $b$, $\Omega$, $k$ and $T$ such that
\begin{equation*}
 \supp \weps(t,\cdot)\subset \left\{x\in \Omega \, \, : \, |b(x)-b(z^0)|\leq \frac{C_{k,T}}{|\ln\eps|^k}\right\}
\end{equation*}
for all $\varepsilon\in (0,\varepsilon_{k,T}]$ and all $t\in [0,T]$.
\end{proposition}

Note that for the particular choice $b(z,r)=r$, corresponding to the axisymmetric 3D Euler equations without swirl, the identical localization property is shown in \cite[Equation (3.8)]{Mar3}. Once, Lemma~\ref{lemma-step} being the adaptation of \cite[Lemma 3.4]{Mar3} is proven, the proof of Proposition~\ref{prop:strong-loc} follows the same lines as the \cite[Proof of (3.8), p.70]{Mar3}. For the sake of a concise exposition, we refer to \cite{Mar3} for full details of the continuity argument, only based on the inequality in Lemma~\ref{lemma:Rt} and the limit of Lemma~\ref{lemma-step}.

\section{Limiting trajectory}\label{sec:trajectory}

Throughout this section, given $T>0$ we consider sharply concentrated initial data $\weps^0$ in the sense of Assumption~\ref{data red} and the unique corresponding weak solution $(\omega_\eps, v_\eps)$ of \eqref{eq:transport-1} on $[0,T]$ with initial data $\weps^0$ and where the exterior field satisfies Assumption~\ref{ass: F}.

We recall that $T_{\eps}$ defined in \eqref{def:maxi-T} satisfies $T_{\eps}=T$ by virtue of Proposition~\ref{prop:strong-loc} and that hence all statements in Sections~\ref{sec:energy} and \ref{sec:momemtum} hold true on $[0,T]$.

\medskip

The purpose of the following proposition is to establish the limiting motion of the point vortex.
\begin{proposition}
\label{prop:limiting-trajectory}
Let $z_\eps$ defined by \eqref{def:point-vortex}. There exists $z\in C^1([0,T], \R^2)$ such that $z_\eps$ converges uniformly to $z$ on $[0,T]$, where $z$ satisfies the ODE
\begin{equation}\label{EDO:limit}
\dot{z}(t)=-\frac{\gamma}{4\pi}\frac{\nabla^\perp b(z(t))}{b(z(t))},\quad z(0)=z^0.
\end{equation}
\end{proposition}

Note that for all $z^0\in \Omega$ as in Assumption~\ref{ass: main part}, see also Assumption~\ref{data red}, there exists a unique global solution $z$ of \eqref{EDO:limit}, see Remark~\ref{rem:ODE}. The limiting dynamics \eqref{EDO:limit} is consistent with the one obtained in \cite{DekeyserVanS, richardson}, see Remark~\ref{rem:ODE-DVS-P}.

\begin{proof}
By Equation \eqref{eq:derivee-z} for the derivative of $z_\eps$ and the expression \eqref{eq:uL} for $v_{\varepsilon,L}$, we have
\begin{align*}
\lne \dot{z}_\eps(t)=&\frac{1}{\gamma_{\varepsilon}} \int_{\Omega}\frac{\nabla^\perp b(x)}{4\pi b^2(x)}\psi_\eps(t,x)(b\omega_\eps)(t,x)\,\dd x +\mathcal{O}(1)\\
=&\frac{1}{\gamma_{\varepsilon}} \frac{\nabla^\perp b(z_\eps(t))}{4\pi b^2(z_\eps(t))}\int_{\Omega}\psi_\eps(t,x)(b\omega_\eps)(t,x)\,\dd x\\
&+\frac{1}{4\pi\gamma_{\varepsilon}} \int_{\Omega}\left(\frac{\nabla^\perp b(x)}{b^2(x)}-\frac{\nabla^\perp b(z_\eps(t))}{ b^2(z_\eps(t))}\right) \psi_\eps(t,x)(b\omega_\eps)(t,x)\,\dd x +
\mathcal{O}(1).
\end{align*}
On the one hand, we note by Propositions~\ref{prop:momentum-bis} and \ref{prop:strong-loc} that $z_{\eps}(t),x\in \cal C_{r_{b}} \subset \Omega_{r_0} $ for all $x\in \supp(\weps(t,\cdot))$ and $t\in [0,T]$. In particular, there exists $C=C(r_0)>0$ such that $|\nabla^2 b(x)|$ and $b^{-1}(x)$ are bounded by $C(r_0)$ on $\Omega_{r_0/2}$. It then follows from the mean-value theorem \eqref{MVT}, Lemma~\ref{lem:est-psi} providing a bound for $\psi_\eps$ and by the estimate on $\Ie$ stated in Proposition~\ref{prop:momentum-bis} that
\begin{align*}
\Bigg| \int_{\Omega}\Big(\frac{\nabla^\perp b(x)}{b^2(x)}-&\frac{\nabla^\perp b(z_\eps(t))}{b^2(z_\eps(t))}\Big) \psi_\eps(t,x)(b\omega_\eps)(t,x)\,\dd x \Bigg| \\
&\leq C \int_{\Omega} |x-z_\eps(t)| |\psi_\eps(t,x)|(b\omega_\eps)(t,x)\,\dd x\\
&\leq C \lne \sqrt{\Ie} 
\leq C\sqrt{\lne},
\end{align*}
where we have used the Cauchy-Schwarz inequality in the second inequality.
On the other hand, Lemma~\ref{lemma:energy-1} and the expansion of the energy given by Proposition~\ref{prop:energy} yield that 
\begin{align*}
\int_{\Omega}\psi_\eps(t,x)(b\omega_\eps)(t,x)\,\dd x&=-2\pi E_{\varepsilon}(t)+\mathcal{O}(1)\\
&=-\gamma_{\varepsilon}^2 b(z^0)\lne +\mathcal{O}(1).
\end{align*}
Therefore, we get
\begin{equation*}
 \dot{z}_\eps(t)=-\frac{\gamma_{\varepsilon}}{4\pi} b(z^0) \frac{\nabla^\perp b(z_\eps(t))}{b^2(z_\eps(t))}+\mathcal{O}\left(\frac{1}{\sqrt{\lne}}\right).
\end{equation*}
Note that in view of Proposition~\ref{prop:momentum-bis} one has that 
\begin{equation*}
 \left|b(z_\eps(t)-b(z^0)\right|\leq \frac{C}{\sqrt{|\ln(\eps)|}},
\end{equation*}
for all $t\in [0,T]$. In particular, there exists $c>0$ such that $c\leq b(z_{\eps}(t))\leq \frac{1}{c}$ for all times and $\eps$ sufficiently small.
We obtain that 
\begin{align*}
\dot{z}_\eps(t)=&-\frac{\gamma_{\varepsilon}}{4\pi} \frac{\nabla^\perp b(z_\eps(t))}{b(z_\eps(t))}-\frac{\gamma_{\varepsilon}}{4\pi}\frac{b(z_0)-b(z_\eps(t))}{b(z_\eps(t))} \frac{\nabla^\perp b(z_\eps(t))}{b(z_\eps(t))}+\mathcal{O}\left(\frac{1}{\sqrt{\lne}}\right)\\
=&-\frac{\gamma_{\varepsilon}}{4\pi} \frac{\nabla^\perp b(z_\eps(t))}{b(z_{\eps}(t))}+\mathcal{O}\left(\frac{1}{\sqrt{\lne}}\right).
\end{align*}

By Ascoli-Arzela, there exists $z$ such that (up to a subsequence still denoted the same way), $z_\eps$ converges to $z$ uniformly on $[0,T]$ with $z$ the unique solution of \eqref{EDO:limit}. By uniqueness of this solution, see Remark~\ref{rem:ODE}, we infer that the full sequence converges to $z$. The conclusion follows.
\end{proof}

\begin{remark}\label{rem:ODE-DVS-P}
In the absence of an exterior field, i.e. $F_{\eps}=0$ and for non-vanishing topographies $b$, the limiting dynamics for a single vortex has formally been derived in \cite{richardson} and rigorously proven by {\sc J. Dekeyser} and {\sc J. Van Schaftingen} \cite{DekeyserVanS}. Coming back to the proof above we observe
\begin{align*}
\dot{z}_\eps(t)&=-\frac{1}{2\gamma_{\varepsilon}} \frac{\nabla^\perp b(z_\eps(t))}{b^2(z_\eps(t))}\frac{E_\eps(t)}{\lne}+\mathcal{O}\left(\frac{1}{\sqrt{\lne}}\right).
\end{align*}
Therefore, this shows that
\begin{equation*}
z_\eps\left(\frac{\gamma \lne s}{E_\eps}\right)\to q(s),
\end{equation*}
where
\begin{equation*}
\dot{q}(s)=\frac{1}{2}\nabla^{\perp}(b^{-1})(q(s)).
\end{equation*}
Hence the asymptotic ODE \eqref{EDO:limit} is consistent with \cite[Theorem 1.1]{DekeyserVanS}. Note that the limiting ODEs differ by the constant $-1/2$ which is due to the definition of $\nabla^{\perp}$ (in the present paper $\nabla^{\perp}b=(-\partial_2b,\partial_1 b)$) and the definition of the energy (where we do not have multiplied by $1/2$ the integral). We refer the reader to \cite[Equation (1.3)]{DekeyserVanS} for a comparison with the dynamics derived in \cite{richardson}.
\end{remark}

The weak localization property provided by Proposition~\ref{prop:weak-loc} with $T_{\eps}=T$ corresponds to {\it(ii)} Theorem~\ref{thm: main red system} whereas the previous proposition is related to {\it(iii)}. The proof of Theorem~\ref{thm: main red system} is then complete.

\appendix

\section{Proof of Proposition~\ref{prop:transport}}
\label{sec:transport-appendix}

The purpose of this section is to gather several properties of the linear transport and continuity equations associated to the nonlinear lake equations in order to prove Proposition~\ref{prop:transport}. The theory of transport equations for non smooth velocity field with bounded divergence has been widely investigated since the pioneering work of {\sc R. J. Di Perna} and {\sc P. L. Lions} \cite{Dip}. 
More recently, existence and stability of renormalized solutions for the lake equations have been proved by {\sc D. Bresch} and {\sc P.-E. Jabin} \cite{BJ}.

Here we will show that any weak solution to the linear transport equation is renormalized, therefore unique, and that it is transported by the flow associated to the velocity field. To this aim, we will adapt the theory of Di Perna and Lions to the present case with possibly unbounded divergence on the boundary, by relying on a specific analysis performed by {\sc B. Desjardins} \cite{Desjardins} in this situation. 

In all the following, $(\Omega,b)$ is a lake satisfying Assumption~\ref{assum:lake} and $\omega^0\in L^\infty(\Omega)$. In the definition of $b$, replacing $c(x)$ by $(2\|\varphi\|_{L^\infty })^\alpha c(x)$ if necessary, we may assume that $\varphi\leq 1/2$ on $\overline{\Omega}$ if $\alpha>0$. If $\alpha=0$, we can also assume $\varphi\leq 1/2$ because it appears only in the definition of $\Omega$ and $\partial\Omega$.

We begin by stating that $\varphi$ is equivalent to the distance to the boundary in the neighborhood of the boundary. This is probably well-known but we provide the full arguments for sake of completeness (and since we did not find any precise reference).

\begin{lemma}\label{lem:d-phi}
There exists $C>0$ such that
\[
\frac1C\dist(x,\partial\Omega)\leq \varphi(x) \leq C\dist(x,\partial\Omega)\quad \text{for all } x\in \Omega.
\]
\end{lemma}
\begin{proof}
We first prove that there exist $\delta>0,C>0$ such that
\begin{equation}\label{ineq:d-phi}
\frac1C\dist(x,\partial\Omega)\leq \varphi(x) \leq C\dist(x,\partial\Omega)\quad \text{for all } x\in \Omega\setminus \Omega_{\delta},
\end{equation}
where we recall that $\Omega_\delta=\{x\::\:\text{dist}(x,\partial\Omega)\geq \delta\}$ according to the definition \eqref{eq:Omega delta}.

The right-hand side inequality in \eqref{ineq:d-phi} is clear by the $C^1$ regularity of $\partial\Omega$, which implies that there exists $\delta_{1}>0$ such that we can define the orthogonal projection $p(x)$ onto $\partial\Omega$ for all $x\in \Omega\setminus \Omega_{\delta_1}$. Hence by the mean value theorem we get
\[
\varphi(x) = |\varphi(x)-\varphi(p(x))|\leq \max_{y\in [p(x),x]} |\nabla \varphi(y)| |x-p(x)| \leq \| \nabla \varphi \|_{L^\infty} \dist(x,\partial\Omega).
\]

For the left-hand side inequality in \eqref{ineq:d-phi}, we use $\nabla \varphi\neq 0$ on $\partial\Omega$, hence by continuity of $\varphi$ and compactness of $\partial\Omega$, there exists a simply connected compact set $K_{1}\Subset \Omega$ and $\delta_{2}>0$ such that $|\nabla \varphi|\geq \delta_{2}$ on $\Omega\setminus K_{1}$. As $ \varphi = 0$ on $\partial\Omega$
 and $\varphi>0$ on $\Omega$, there exists a simply connected compact set $K_{1}\Subset K_{2}\Subset \Omega$ such that $\min_{\partial K_{1}} \varphi > \sup_{\Omega\setminus K_{2}}\varphi$.

For any $x\in \Omega\setminus K_{1}$, we consider $Y(\cdot,x)\in C^1([0,T_{x}), \Omega\setminus K_{1})$ the unique solution of 
\[
\frac{\dd}{\dd s} Y(s,x)=-\frac{\nabla \varphi(Y(s,x))}{|\nabla \varphi(Y(s,x))|^2},\quad Y(0,x)=x,
\]
where $T_{x}$ is either infinite or corresponds to the time when the trajectory reaches the boundary: $Y(T_{x},x)\in \partial(\Omega\setminus K_{1})= \partial\Omega \cup \partial K_{1}$.

For all $x\in \Omega\setminus K_{2}$, we state that $T_{x}=\varphi(x)$ and that $Y(T_{x},x)\in \partial\Omega$. Indeed, we have $\displaystyle \frac{\dd}{\dd s} \varphi(Y(s,x))=-1$ on $(0,T_{x})$ so the function $s\mapsto \varphi(Y(s,x))$ is a decaying function, and it is not possible that $Y(T_{x},x)\in \partial K_{1}$ because $\min_{\partial K_{1}} \varphi > \varphi(x)$. So the only restriction on $T_{x}$ is to reach $\partial\Omega$ which is the case only when $T_{x}=\varphi(x)$, because we recall that $\Omega=\{\varphi\neq 0\}$ and $\partial\Omega=\{\varphi= 0\}$.

Considering this trajectory implies that for all $x\in \Omega\setminus K_{2}$, we have 
\begin{equation*}
\dist(x,\partial\Omega)\leq |x-Y(T_{x},x)|\leq \int_0^{T_x=\varphi(x)}\frac{ds}{|\nabla \varphi(Y(s,x))|}\leq \frac1{\delta_{2}} \varphi(x).
\end{equation*}
We end the proof of \eqref{ineq:d-phi} by considering $\delta=\min (\delta_{1},\delta_{3})$ where $\delta_{3}>0$ is small enough such that $K_{2}\subset \Omega_{\delta_{3}}$.

We have just proved that the functions $x\mapsto \frac{\varphi(x)}{\dist(x,\partial\Omega)}$ and $x\mapsto \frac{\dist(x,\partial\Omega)}{\varphi(x)}$ are bounded in a neighborhood $\Omega\setminus \Omega_\delta$ of the boundary, but it is clear that they are also bounded in $\Omega_{\delta}$, this finishes the proof of the lemma.
\end{proof}

\begin{proposition}\label{prop:transport-appendix}
Let $T>0$ and let $(v,\omega)$ be the weak solution of the lake equations as in Proposition~\ref{prop:well-posed} with initial condition $\omega^0$ on $[0,T]$. For this velocity field $v$, consider the linear transport equation
\begin{equation}\label{eq:transport-appendix}
\partial_t \rho +v \cdot \nabla \rho=0, \quad \rho(0)=\omega^0,\quad \text{on } \Omega
\end{equation}
and the linear continuity equation 
\begin{equation}\label{linear-cont}
 \partial_t(b\rho)+\diver(b v \rho)=0,\quad \rho(0)=\omega^0 \quad \text{on } \Omega.
\end{equation}
Then $\omega$ is the unique distributional bounded solution on $[0,T]\times \Omega$ of \eqref{eq:transport-appendix} and \eqref{linear-cont}\footnote{Note that here, we consider test functions that are compactly supported on $\Omega$, while the test functions are not necessarily compactly supported in the weak formulation given by Proposition~\ref{prop:well-posed}}. Moreover, $|\omega|^p$ is also a distributional bounded solution for any $p> 1$.

Finally, we have 
$\|b^{1/p}\omega(t,\cdot)\|_{L^p}=\|b^{1/p}\omega^0\|_{L^p}$ for $t\in [0,T]$ and $p\in [1,\infty)$.
\end{proposition}

\begin{proof}
Let $\delta>0$ defined in \eqref{ineq:d-phi}. For all $x\in \Omega\setminus \Omega_{\delta}$, by the proof of Lemma~\ref{lem:d-phi}, we may define the orthogonal projection $p(x)$ onto $\partial\Omega$. We write
\begin{multline*}
 \nabla \varphi(x)\cdot v(t,x)
= \left(v(t,x)-v(t,p(x)\right)\cdot \nabla \varphi(x)\\
+v(t,p(x))\cdot \left(\nabla \varphi(x)-\nabla \varphi(p(x))\right)
\end{multline*}
where we have used the fact that $v\cdot \textbf{n} = 0= v\cdot \nabla \varphi$ on $\partial\Omega$. Therefore we obtain by the log-Lipschitz regularity of $v$ and the Lipschitz regularity of $\nabla \varphi$ that for all $(t,x)\in [0,T]\times (\Omega\setminus \Omega_{\delta})$:
\begin{align*}
|\nabla \varphi(x)\cdot v(t,x)|
&\leq C |x-p(x)|\left(1+\left|\ln |x-p(x)|\right|\right)\\
&\leq C \varphi(x)(1+|\ln \varphi(x)|)=C \varphi(x)(1-\ln \varphi(x)) \\
&\leq -C \varphi(x) \ln \varphi(x) ,
\end{align*} where $C$ depends on $\|\omega\|_{L^\infty}$
 and where we have used that $\varphi\leq 1/2$. As the function $x\mapsto -\varphi(x) \ln \varphi(x)$ is greater than a positive constant on $\Omega_{\delta}$, we have that the previous inequality holds also true in $\Omega_{\delta}$:
\begin{equation}\label{eq:app}
|\nabla \varphi(x)\cdot v(t,x)|
\leq -C \varphi(x) \ln \varphi(x) , \quad \forall (t,x)\in [0,T]\times \Omega.
\end{equation}

Recalling $b=c\varphi^\alpha$ so $\nabla b= \varphi^\alpha \nabla c+\alpha c \varphi^{\alpha-1}\nabla \varphi$, the first consequence of \eqref{eq:app} is to estimate $\diver v$ as follows
 \begin{align*}
 | \diver v(x)|&=\Big|\frac{\nabla b(x)}{b(x)} \cdot v(x)\Big|
\leq C+C\Big|\frac{\nabla \varphi(x)}{\varphi(x)} \cdot v(x)\Big|
\end{align*}so that
 \begin{align}\label{ineq:bounded-div}
 | \diver v(x)| &\leq C (1 -\ln \varphi(x)) \leq C (1 -\ln \dist(x,\partial\Omega)),
\end{align}
from which it follows that
 \begin{equation*}
\exp(T_0|\diver v(t,x)|) \leq \exp(CT_0) \dist(x,\partial\Omega)^{-CT_0}\in L^1([0,T]\times \Omega)
\end{equation*}
 for $T_0$ sufficiently small.
We may therefore apply the result by {\sc B. Desjardins} \cite[Lemma 3]{Desjardins}: the {linear} transport equation \eqref{eq:transport-appendix}
has a unique distributional bounded solution on $[0,T]\times \Omega$ and this solution is renormalized (see the remark just after \cite[Theorem 3]{Desjardins}): for all $\beta\in C^1(\R)$ bounded, $\beta(\omega)$ is also a weak bounded solution.

We infer that the corresponding {linear} continuity equation \eqref{linear-cont}
also has the renormalization property on $[0,T]\times \Omega$ since the distributional formulations of these equations are equivalent on $[0,T]\times\Omega$ (noticing that $b>0$ and $b\in C^1$ on $\Omega$).

Let $p\in (1,\infty)$. Setting $\beta\in C^1$ such that $\beta(s)=s^p$ for $|s|\leq \|\omega\|_{L^\infty}$ and $\beta$ constant for $|s|\geq 2\|\omega\|_{L^\infty}$, hence $\beta(\omega)=|\omega|^p$, we obtain that $|\omega|^p$ satisfies \eqref{linear-cont} in the sense of distributions on $[0,T]\times \Omega$.

Finally, we introduce a smooth, non-increasing function $\chi_0:\R \to [0,1]$ which is identically one on $[0,1/4]$ and vanishes on $[1/2,1)$ and we take as a test function 
$$
\Phi_R(x)=(1-\chi_0) \left(R\varphi(x)\right)
$$
which is compactly supported in $\Omega$ for all $R>0$. Observe that for $1/4\leq R\varphi(x)\leq 1/2$ we have by using \eqref{eq:app}
\begin{align*}
b(x)|\omega|^p(t,x)|v(x)\cdot \nabla \Phi_R(x)|&\leq R b(x)|\omega|^p(t,x)| |v(x)\cdot \nabla \varphi(x)| |\chi_0'(R\varphi(x))|\\
&\leq C\|\omega\|_{L^\infty}^p \|b\|_{L^\infty} (R \varphi(x)) |\ln \varphi(x)| |\chi_0'(R\varphi(x))| \\
& \leq C | \ln R| \mathds{1}_{\dist(x,\partial\Omega)\leq C/R}
\end{align*}
which implies that
\[
\int b(x)|\omega|^p(t,x)|v(x)\cdot \nabla \Phi_R(x)|\, \dd x\leq \frac{C|\ln R|}R \to 0
\]
as $R\to \infty$.
Therefore $\int_\Omega b(x)|\omega|^p(t,x)\, \dd x = \int_\Omega b(x)|\omega^0|^p(x)\, \dd x.$

As this equality holds for any $p>1$, it is enough to consider a sequence $p_{n}=1+\frac1n$ to state that it holds also true for $p=1$, just by using the dominated convergence theorem.
\end{proof}

\begin{proposition}\label{prop:support-appendix}
Let $\omega^0\in L^\infty(\Omega)$ with compact support in $\Omega$. Let $(v,\omega)$ be the unique weak solution of the lake equations as in Proposition~\ref{prop:well-posed} with initial condition $\omega^0$.
There exists a compact subset $K_T$ of $\Omega$, depending only on $\|\omega^0\|_{L^\infty}$, $\delta_{0}$ and $T$, such that 
\[
\supp \omega(t,\cdot)\subset K_T, \quad \forall t\in [0,T].
\]
\end{proposition}
We recall that $\delta_{0}$ is defined by $\dist(\supp \omega_{0},\partial\Omega)$ in Proposition~\ref{prop:transport}.
\begin{proof} We set
$$
\delta_1:=\min\Big(\frac{e^{-1}}2\ ;\ \inf_{\supp \omega^0 } \varphi \Big).
$$

By Proposition~\ref{prop:transport-appendix}, $\omega^2$ is a distributional bounded solution of \eqref{linear-cont} with initial condition $(\omega^0)^2$. In view of the time-regularity properties of $\omega$ stated in Proposition~\ref{prop:well-posed}, we obtain that $\omega^2$ satisfies the following weak formulation for all $t\in \R_+$ and for all test function $\Phi\in C^1_c([0,T]\times \Omega)$:
\begin{multline*}
\int_{\Omega} \Phi(t,x) (b\omega^2)(t,x) \,\dd x - \int_{\Omega} \Phi(0,x) (b(\omega^0)^2)(x) \,\dd x \\
 = \int_{0}^t \int_{\Omega} ( b\omega^2)(s,x) (\partial_t \Phi+v\cdot \nabla \Phi)(s,x) \,\dd x\dd s.
 \end{multline*}
Moreover, we may assume that this formulation holds for any test function $\Phi\in C^1([0,T]\times \Omega)$: this is established by arguing as in the proof of Proposition~\ref{prop:transport-appendix} namely replacing $\Phi$ by $\Phi \Phi_{R}\in C_c^1$ and letting $R\to \infty$. 

Now we introduce as before a smooth, non-increasing function $\chi_0:\R \to [0,1]$ which is identically one on $[0,1/4]$ and vanishes on $[1/2,1)$ and we set
\[
\chi(t,x)=\chi_0\left( \frac{\varphi(x)}{r(t)}\right)
\]
with $r(t)$ a decreasing $C^1$ function less than one to be determined later on. 
Using the weak formulation for $\omega^2$ with test function given by $\chi$, we obtain for any $t\in \R_+$ 
\begin{align*}
&\int_{\Omega} \chi(t,x) (b\omega^2)(t,x) \,\dd x - \int_{\Omega} \chi(0,x) (b(\omega^0)^2)(x) \,\dd x \\
& = \int_{0}^t \int_{\Omega} ( b\omega^2)(s,x) \chi_0'\left( \frac{\varphi(x)}{r(s)}\right)\frac{1}{r(s)}\left(-\frac{r'(s)}{r(s)}\varphi(x) + \nabla \varphi(x)\cdot v(s,x) \right) \,\dd x\dd s\\
&= \int_{0}^t \int_{\Omega} ( b\omega^2)(s,x) \left|\chi_0'\left( \frac{\varphi(x)}{r(s)}\right)\right|\frac{1}{r(s)}\left(\frac{r'(s)}{r(s)}\varphi(x) -\nabla \varphi(x)\cdot v(s,x) \right) \,\dd x\dd s.
 \end{align*}
We infer from \eqref{eq:app} that for $\frac{r(s)}{4}\leq \varphi(x)\leq \frac{ r(s)}{2}$
\[
\frac{r'(s)}{r(s)}\varphi(x) -\nabla \varphi(x)\cdot v(s,x) \leq \frac{r'(s)}{4} +C r(s)(1-\ln r(s)).
\]
Setting 
\[
r(s)=\exp\Big(- e^{8C(s+s_{0})}\Big)
\]
where $s_{0}\geq 0$ is chosen such that $r(0)= 2 \delta_1$, which uniquely exists because $2\delta_{1} \leq e^{-1}$, namely $s_{0}=\frac1{8C}\ln(-\ln(2\delta_{1}))$. With this expression, it is clear that $r$ is decaying, less than one, and verifies for any $s\geq 0$
\[
 \frac{r'(s)}{4} +C r(s)(1-\ln r(s)) =Cr(s) \Big( -2 e^{8C(s+s_{0})}+ 1 +e^{8C(s+s_{0})}\Big)\leq 0,
\]
where we have used that $s_{0}\geq 0$. For such a function $r$, this implies
\[
\int_{\Omega} \chi(t,x) (b\omega^2)(t,x) \,\dd x \leq \int_{\Omega} \chi(0,x) (b(\omega^0)^2)(x) \,\dd x =0
\]
because for any $x\in \supp \omega^0$, $\varphi(x)\geq \delta_{1} =r(0)/2$, hence $ \chi(0,x)=0$.
From this inequality, we deduce that 
$\supp \omega(t,\cdot)\subset \{x : \varphi(x)\geq \frac{1}{4}r(t)\}.$ 
The conclusion follows from Lemma~\ref{lem:d-phi}.

Note that such arguments establishing support properties without requiring the notion of flow as been already used by the two last authors in \cite{lacave-miot}.
\end{proof}

Next, the regularity assumptions on $v$ in $[0,T]\times \Omega$ allow one to define the flow in the classical sense: for any $x \in \Omega$ and $t_0$, there exist $0\leq t_{1} <t_{0}<t_{2}\leq T$ and a unique characteristic curve $X(\cdot, t_0,x)\in C^1([t_{1},t_{2}] ; \Omega)$ solving
\[
\frac{\dd X(t,t_0,x)}{\dd t}=v(t,X(t,t_0,x)), \quad X(t_0,t_0,x)=x.
\]

To prove that we can choose $t_{1}=0$ and $t_{2}=T$, we need to establish that the trajectory starting from $x\in \Omega$ cannot reach the boundary in finite time. 

\begin{proposition}\label{prop:distance-evol}
Let $(v,\omega)$ be the weak solution of the lake equations as in Proposition~\ref{prop:well-posed} with initial condition $\omega^0$. There exists $\beta>1$ depending only on $\|\omega^0\|_{L^\infty}$ and $T$ such that
for any $x\in \Omega$
\[
\varphi(X(t,t_0,x))\geq \varphi(x)^{\beta}, \quad \forall t_{0}\in [0,T],\ \forall t\in [t_{1},t_{2}].
\]
\end{proposition}

\begin{proof}

This property is proved in \cite[Lemma 4.1]{AlTakiLacave} in the case of smooth vorticity, and we provide a self-contained proof below.

We compute again
\[
\frac{\dd}{\dd t} \varphi(X(t,t_0,x))=v(t,X(t,t_0,x))\cdot \nabla \varphi(X(t,t_0,x))
\]
and therefore we obtain by \eqref{eq:app}
\begin{equation}\label{ine:distance-evol}
\Big|\frac{\dd}{\dd t} \varphi(X(t,t_0,x))\Big| \leq - C\varphi(X(t,t_0,x))\ln \varphi(X(t,t_0,x)),\quad t\in [t_{1},t_{2}] 
\end{equation}
which implies that
\[
\frac{\dd}{\dd t} \ln \Big(- \ln \varphi(X(t,t_0,x))\Big) \leq C,\quad t\in [t_{1},t_{2}]
\]
hence
\[
\varphi(X(t,t_0,x))\geq \varphi(x)^{e^{CT}} ,\quad t\in [t_{0},t_{2}].
\]
On the other hand, \eqref{ine:distance-evol} also gives
\[
\frac{\dd}{\dd t} \ln \Big(- \ln \varphi(X(t,t_0,x))\Big) \geq -C,\quad t\in [t_{1},t_{2}]
\]
hence
\[
\varphi(X(t,t_0,x))\geq \varphi(x)^{e^{CT}} ,\quad t\in [t_{1},t_{0}]
\]
from which the conclusion follows.
\end{proof}

This proposition means that we extend the solution in order to get $X(\cdot, t_0,x)\in C^1([0,T] ; \Omega)$. By uniqueness, we have that $X(t,t_0,\cdot)$ is an homeomorphism for any $t,t_0\in [0,T]$, with $X(t,t_0,\cdot)^{-1}=X(t_0,t,\cdot)$.

Note that at this stage we never use that $\omega$ is transported by the flow, which is not obvious due to the singularity of $\diver v$ on the boundary. This property is the content of the following:

\begin{proposition}\label{prop:last}
Let $\omega^0\in L^\infty(\Omega)$ with compact support in $\Omega$.
Let $(v,\omega)$ be the unique weak solution of the lake equations as in Proposition~\ref{prop:well-posed} with initial condition $\omega^0$. Then we have $\omega(t,y)=\omega^0(X(0,t,y))$ for all $y\in \Omega$ and $t\in [0,T]$.
\end{proposition}

\begin{proof}
Let $\Omega_0 \Subset\Omega$ such that $\supp (\omega^0)\subset \Omega_0$.
By Proposition~\ref{prop:support-appendix} there exists $\Omega_1\Subset \Omega$, with $\Omega_0\subset \Omega_1$, such that $\supp \omega(t,\cdot)\subset \Omega_1$ for all $t\in [0,T]$. By Proposition~\ref{prop:distance-evol}, increasing $\Omega_1$ if necessary, we may assume that
\begin{equation}X(t,0,x)\in \Omega_1,\quad \forall x\in \Omega_0, \quad \forall t\in [0,T]. \label{eq:inclusion}\end{equation}

Moreover, by Proposition~\ref{prop:distance-evol} we may introduce $\Omega_1 \subset \Omega_2 \Subset \Omega$ such that
\begin{equation}X(t,t_0,x)\in \Omega_2,\quad \forall x\in \Omega_1, \quad \forall t,t_0\in [0,T]. \label{eq:inclusion-2}\end{equation}

We finally introduce $\Omega_2 \Subset \Omega_3 \Subset \Omega$, 
we extend all functions by $0$ outside $\Omega$ and we set 
$\widetilde{v}=v\chi$ where the smooth function $\chi$ satisfies $\chi =1$ on $\Omega_2$ and $\chi=0$ in $\Omega_3^c$.

By definition of $\Omega_1$ we infer that $\omega$ is a weak bounded solution of
$$\partial_t \omega + \widetilde{v}\cdot \nabla \omega=0 \quad \text{on }[0,T]\times \R^2$$
where $\diver(\widetilde{v})=\diver(v)\chi+v\cdot \nabla \chi \in L^\infty([0,T]\times \R^2)$ by \eqref{ineq:bounded-div} and the definition of $\chi$, and where $\widetilde{v}$ satisfies the same regularity as $v$. In particular, we have
$\widetilde{v}\in L^1([0,T];W^{1,1}(\R^2))$ and $\widetilde{v}/(1+|x|)\in L^1([0,T];L^1(\R^2))+L^1([0,T];L^\infty(\R^2))$.
We may then invoke classical results on linear transport equations on the full space with vector fields of bounded divergence that where established by {\sc R. J. Di Perna} and {\sc P. L. Lions} \cite[Theorem III.2]{Dip}: we have
\begin{equation*}
\omega(t,y)=\omega^0(\widetilde{X}(0,t,y)),\quad \forall x\in \R^2, \quad \forall t\in [0,T]
\end{equation*}
where $\widetilde{X}$ is the flow associated to $\widetilde{v}$. In particular this holds true for $y\in \Omega$. Now we fix $t\in [0,T]$ and $y\in \Omega$. There are two cases.

\medskip

\textbullet \: If $y\in \Omega\setminus \Omega_1$, we have $\omega(t,y)=0$ by definition of $\Omega_1$. Let $x=X(0,t,y)$. If $x\in \text{supp}(\omega^0)$ then $x\in \Omega_0$ thus by \eqref{eq:inclusion}, we have $y=X(t,0,x)\in \Omega_1$ which is a contradiction. Therefore $x\notin \text{supp}(\omega^0)$ and $ \omega(t,y)=0=\omega^0(x)=\omega^0(X(0,t,y))$.

\medskip

\textbullet \: If $y\in \Omega_1$, we have $y=\widetilde{X}(t,0,\widetilde{x})$ (where $\widetilde{x}=\widetilde{X}(0,t,y)$) so that $\omega(t,y)=\omega^0(\tilde{x})$. Let $f(s)=\widetilde{X}(s,t,y)$ and $g(s)={X}(s,t,y)$ so that $f(t)=g(t)=y\in \Omega_1$. By \eqref{eq:inclusion-2} we have $g(s)\in \Omega_2$ for all $s\in [0,T]$. On the other hand, by continuity, there exists $(t_1,t_2)\in [0,T]$ containing $t$, maximal such that $f(s)\in \Omega_2$ on $(t_1,t_2)\in [0,T]$. Thus on $(t_1,t_2)$ we have $\tilde{v}(s,f(s))=v(s,g(s))$ therefore $f$ and $g$ coincide on $(t_1,t_2)$. Since $g\in \Omega_2$ on $[0,T]$ we conclude that $[t_1,t_2]=[0,T]$ and finally $f=g$ on $[0,T]$. In particular, $f(0)=g(0)$ which means that $\widetilde{x}=X(0,t,y)$. Thus, we have obtained 
$\omega(t,y)=\omega^0(X(0,t,y))$. Proposition~\ref{prop:last} is proved.
\end{proof}

\section{Rearrangement of the mass}\label{app:rearrangement}

 For $M_{0},\gamma>0$, we define
\[
F_{M_{0},\gamma} := \Big\{ f\in L^\infty_{c}(\R^2),\, 0\leq f \leq M_{0}, \, \int f=\gamma \Big\}.
\]
The following rearrangement result is commonly used in the literature, but we provide a proof for sake of completeness.
\begin{lemma}\label{lem:rearrang0}
 Let $g$ be a non increasing continuous function from $(0,\infty)$, non-negative, such that $s\mapsto s g(s)\in L^1_{\loc}([0,\infty))$. Then for all $x\in \R^2$, we have
 \[
 \sup_{f\in F_{M_{0},\gamma}} \int_{\R^2} g(|x-y|) f(y) \, \dd y = 2\pi M_{0} \int_0^{R_{0}} sg(s) \,\dd s \text{ with }R_{0}=\sqrt{\frac{\gamma}{\pi M_{0}}},
 \]
 i.e. the supremum is reached for the function $f^\ast=M_{0}\mathds{1}_{B(x,R_{0})}$.
\end{lemma}

\begin{proof} 
Let $f\in F_{M_{0},\gamma}$ given.
We have
\begin{align*}
\int_{\R^2} &g(|x-y|) f(y) \, \dd y=\int_{B(x,R_0)} g(|x-y|) f(y) \, \dd y+\int_{\R^2\setminus B(x,R_0)} g(|x-y|) f(y) \, \dd y\\
=&\int_{B(x,R_0)} g(|x-y|) (f(y)-f^\ast(y)) \, \dd y+\int_{\R^2\setminus B(x,R_0)} g(|x-y|) f(y) \, \dd y \\
&+\int_{B(x,R_0)} g(|x-y|) f^\ast(y) \, \dd y.
\end{align*}
For $y\in B(x,R_0)$ we have $f(y)\leq f^\ast(y)$ and $g(|x-y|)\geq g(R_0)$, therefore
\begin{align*}
\int_{B(x,R_0)} g(|x-y|) (f(y)-f^\ast(y)) \, \dd y &\leq g(R_0)\int_{B(x,R_0)}(f(y)-f^\ast(y)) \, \dd y
\\
&\leq g(R_0)\int_{B(x,R_0)}f(y)\, \dd y -g(R_0)\gamma.
\end{align*}

Next, for $y\in \R^2\setminus B(x,R_0)$ we have $g(|x-y|)\leq g(R_0)$, therefore
\begin{equation*}
\int_{\R^2\setminus B(x,R_0)} g(|x-y|) f(y) \, \dd y\leq g(R_0)\int_{\R^2\setminus B(x,R_0)} f(y) \, \dd y.
\end{equation*}

Altogether, we obtain
\begin{align*}
\int_{\R^2} g(|x-y|) f(y) \, \dd y \leq& g(R_0)\int_{B(x,R_0)}f(y)\, \dd y -g(R_0)\gamma\\
&+g(R_0)\int_{\R^2\setminus B(x,R_0)} f(y) \, \dd y
+\int_{B(x,R_0)} g(|x-y|) f^\ast(y) \, \dd y\\
\leq& g(R_0)\int_{\R^2}f(y)\, \dd y -g(R_0)\gamma +\int_{B(x,R_0)} g(|x-y|) f^\ast(y) \, \dd y\\
\leq & \int_{\R^2} g(|x-y|) f^\ast(y) \, \dd y.
\end{align*}
The conclusion follows.
\end{proof}



\end{document}